%% file: main.tex
\documentclass[11pt]{article}
\usepackage[a4paper,margin=26mm]{geometry}
\usepackage[T1]{fontenc}
\usepackage{lmodern,amsmath,amssymb,amsthm,mathtools,booktabs,microtype}
\usepackage[colorlinks=true,linkcolor=blue,citecolor=blue,urlcolor=blue]{hyperref}
\newtheorem{theorem}{Theorem}[section]
\newtheorem{proposition}[theorem]{Proposition}
\newtheorem{lemma}[theorem]{Lemma}
\newtheorem{corollary}[theorem]{Corollary}
\theoremstyle{definition}
\theoremstyle{remark}\newtheorem{remark}[theorem]{Remark}
\newcommand{\E}{\mathbb E}\newcommand{\Pp}{\mathbb P}
\newcommand{\T}{\mathbb T}\newcommand{\R}{\mathbb R}
\newcommand{\Z}{\mathbb Z}
\newcommand{\dd}{\,\mathrm d}\newcommand{\ind}{\mathbf1}
\newcommand{\Var}{\operatorname{Var}}
\newcommand{\TV}{\operatorname{TV}}
\newcommand{\Ree}{\operatorname{Re}}
\newcommand{\eps}{\varepsilon}
\newcommand{\oo}{o_{\Pp}}
\numberwithin{equation}{section}\allowdisplaybreaks
\title{Spectral extremes under exact cycle conditioning}
\author{Zhipeng Lu\thanks{Shenzhen MSU--BIT University and Guangdong
Laboratory of Machine Perception and Intelligent Computing,
Shenzhen 518172, Guangdong, China.
Email: \href{mailto:zhipeng.lu@hotmail.com}{zhipeng.lu@hotmail.com}.}}
\date{November 30, 2026}
\hypersetup{pdftitle={Spectral extremes under exact cycle conditioning},
pdfauthor={Zhipeng Lu},bookmarksdepth=2}
\begin{document}\maketitle
\begin{abstract}
We study the maximum modulus of the characteristic polynomial of a
random permutation matrix conditional on its number of cycles.
Writing $\kappa=k/\log n$, the leading center is the minimum of
$(\log n+k\lambda(s))/s$, where $\lambda$ is the log-sine cumulant.
The second-order term is $-c(\kappa)\log\log n$.
Its coefficient is expressed through the quenched persistence exponent
of Brownian motion against a Brownian wall. The error is
$o_{\mathbb P_{n,k}}(\log\log n)$, uniformly on positive compact
$\kappa$ intervals. The proof first transfers the full short-cycle
configuration to a harmonic sampling law. A single random-wall
excursion estimate supplies both a coarse low-frequency localization
and the sharp path probabilities. One comparison on positive phase
cells treats separated and nearby angles, while other rational resonances have
a strict half-tilt gap. The resulting second-moment ratio tends to one.
We also obtain a second-order random centering under the Ewens law
and an application to the superposition of two perfect matchings.
\end{abstract}

\section{Introduction}\label{rn:introduction}

Let $C_j$ count the $j$-cycles of a permutation $\sigma_n$ and let $P_n$
be its permutation matrix. With $e(t)=\exp(2\pi it)$ and normalized
Lebesgue measure on $\T=\R/\Z$, put
\begin{equation}\label{rn:model}
 F_n(z)=\det(I-zP_n)=\prod_{j=1}^n(1-z^j)^{C_j},\qquad
 M_n=\log\|F_n\|_\infty,\qquad K_n=\sum_jC_j.
\end{equation}
We write $\Pp_{n,k}$ for the law conditional on $K_n=k$.
For every positive Ewens parameter this is the same law: the factor
$\theta^{K_n}$ in the profile weight is constant after conditioning.

For $U$ uniform on $\T$ let $V=\log|1-e(U)|$. Define
\begin{equation}\label{rn:cumulant}
 A(s)=\E e^{sV}=\frac{\Gamma(1+s)}{\Gamma(1+s/2)^2},
 \qquad \lambda(s)=\log A(s),\qquad s>-1.
\end{equation}
The $s$-tilted law has density $e^{sx-\lambda(s)}$ relative to the
law of $V$. Its mean is $\lambda'(s)$ and its variance is
$\lambda''(s)>0$. For $\kappa>0$, let $s_\kappa$ be the unique solution
\begin{equation}\label{rn:critical}
 \kappa\{s_\kappa\lambda'(s_\kappa)-\lambda(s_\kappa)\}=1.
\end{equation}
Indeed the derivative of $s\lambda'(s)-\lambda(s)$ is
$s\lambda''(s)>0$; its limits at zero and infinity are zero and infinity.
The second limit follows, for example, from
$\lambda(s)=s\log2-\frac12\log s+O(1)$.
Set
\begin{equation}\label{rn:center}
 v(\kappa)=\kappa\lambda'(s_\kappa),\qquad
 m_{n,k}=v(k/\log n)\log n
       =\inf_{s>0}\frac{\log n+k\lambda(s)}s.
\end{equation}

Let $\gamma$ be the Brownian random-wall persistence exponent:
for independent standard Brownian motions $B,W$,
\begin{equation}\label{rn:gamma}
 -\frac{\log\Pp(1+B_t\ge\beta W_t,\ 0\le t\le T\mid W)}{\log T}
 \longrightarrow\gamma(\beta)
\end{equation}
in probability. The persistence theorem of
Mallein and Mi{\l}o\'s \cite{MalleinMilosWall} gives this finite,
continuous, convex function. It satisfies
$\gamma(0)=1/2$ and $\gamma(\beta)>1/2$ for $\beta\ne0$.
The version with an unrestricted endpoint in \eqref{rn:gamma} follows
from their endpoint-restricted theorem by the association argument
given below.
Define
\begin{equation}\label{rn:coefficient}
 \beta_\kappa=\frac1{\kappa s_\kappa\sqrt{\lambda''(s_\kappa)}},
 \qquad \alpha_\kappa=2\gamma(\beta_\kappa)+\frac12,\qquad
 c(\kappa)=\frac{\alpha_\kappa}{s_\kappa}.
\end{equation}

\begin{theorem}\label{rn:main}
For every compact interval $\mathcal K\subset(0,\infty)$ and
$\eps>0$,
\begin{equation}\label{rn:limit}
 \sup_{\substack{k\in\Z\\ k/\log n\in\mathcal K}}
 \Pp_{n,k}\left(
  |M_n-m_{n,k}+c(k/\log n)\log\log n|>\eps\log\log n
 \right)\longrightarrow0.
\end{equation}
\end{theorem}
The supremum is over probabilities under different conditional laws.
No common coupling of those laws is asserted. The function $c$ is
specified by a persistence exponent; an elementary closed formula
for that exponent is not required by the result. In particular,
$c(\kappa)>3/(2s_\kappa)$.
The precision in \eqref{rn:limit} identifies the coefficient of
$\log\log n$. It does not assert an $O_{\Pp}(1)$ remainder or a
limit law after subtracting the displayed center.

Cook and Zeitouni \cite{CZ} established the first-order constant for a
uniform permutation. The present conditional formulation retains the
fluctuating cycle profile as an environment. If $E_i$ is an exponential
gap of rate $\kappa$ and $V_i$ is an independent $s_\kappa$-tilted mark,
the centered path increment is
\begin{equation}\label{rn:walk}
 s_\kappa V_i-\lambda(s_\kappa)-E_i
 =s_\kappa(V_i-\lambda'(s_\kappa))-(E_i-1/\kappa).
\end{equation}
The two variances in this identity have ratio $\beta_\kappa^2$.
Two one-sided persistence costs and a fixed-width middle density
give the exponent $\alpha_\kappa$. This explains the coefficient;
the arithmetic pair estimates are needed to turn the path cost into
a maximum theorem. We do not identify the spectral field with a
branching random walk.
Mallein and Mi{\l}o\'s \cite{MalleinMilosBRW} established the
corresponding random-environment correction for branching random
walks. The present argument uses their random-wall estimates and
proves the comparisons required by the permutation field.

The work began with the two-matching encoding in \cite{Lu}.
Cook and Gu \cite{CG} prove a sharp second-order asymptotic for a
field with continuous Poisson frequencies, independent uniform
phases, and a smooth periodic kernel. Their Conjecture~1.6 concerns
the characteristic polynomial of a uniform random permutation.
Corollary~\ref{rn:ewens} at $\theta=1$ gives precisely that expansion;
the identification of constants is recorded in
Remark~\ref{rn:cg-comparison}. Theorem~\ref{rn:main} also allows
exact cycle counts throughout positive compact density intervals.

Both proofs use the random-environment barrier framework. Here the
additional comparisons address integer frequencies, the log-sine
singularity, and the simultaneous size and cycle-count constraints.
The proof does not invoke the results of Cook and Gu. The earlier
version recorded awareness of their public announcements on
September 10, 2026. That record is retained without a priority claim;
the present comparison concerns their now available preprint.

The proof separates the exact conditioning, the reference path estimate,
and the arithmetic comparison. A coefficient calculation controls the
density of the short configuration on a high-probability set. The density
has total mass one, so this local control already gives total variation
convergence to iid harmonic samples with a random count; no global
coefficient bound is needed. This transfer is used only after the maximum
has been studied under the harmonic law.

For that harmonic model, a single reference excursion theorem supplies
both coarse localization and the sharp path exponent. Its middle bridge
uses a local central limit theorem and a maximal inequality. Phase-cell
mass estimates then connect the arithmetic field to the reference paths:
one positive-tilt comparison treats both separated angles and nearby
angles. Other rational relations require a separate strict half-tilt gap.
The resulting weighted second moment retains the low-frequency center.
Coarse localization is therefore established before the sharp maximum
argument, with no dependence on its conclusion. Finally, total variation,
deterministic insertion of the long cycles, and direct comparison of two
variational minima return the result to the conditional permutation law.

Throughout, $L=\log n$ and $\ell=\log L$ in the permutation model.
An estimate called uniform is uniform for $\kappa$ in a fixed positive
compact interval. Constants can depend on that interval and on a fixed
error tolerance. In a harmonic model the logarithmic horizon is denoted
by $L$ as well; changes of $O(\log\ell)$ between the two horizons do not
affect a coefficient of $\ell$.

\input{conditioning}

\input{reference}

\input{harmonic_core}

\input{arithmetic}

\input{extremes}

\input{consequences}

\section*{Tools and responsibility}
The project started from the author's earlier work \cite{Lu}.
The author supplied the starting questions, mathematical direction,
and review decisions. AI assistants, including Codex, assisted with
derivations, proof reconstruction, drafting, and computational checks.
This version separates the new analytic arguments from the older
formalization. The repository
\url{https://github.com/Lzp88/conditional-spectral-extremes}
contains the earlier Lean development and its recorded evidence.
That evidence does not verify the new persistence transfer, phase-cell
comparison, or second-order theorem in this version. The author remains
responsible for the mathematical claims.

\input{references}
\end{document}

%% file: conditioning.tex
\section{Removing the exact size constraint}\label{rc:section}

Fix a compact interval $K=[\kappa_-,\kappa_+]\subset(0,\infty)$.
Every estimate in this section is uniform over integers $k$ with
$\kappa=k/\log n\in K$. Write
\[
 L=\log n,\qquad \ell=\log L,\qquad
 b=\lfloor n/\ell^2\rfloor,\qquad
 H_b=\sum_{j=1}^b\frac1j,\qquad T=L-H_b.
\]
Thus $T=2\log\ell-\gamma_{\rm E}+o(1)\to\infty$, where $\gamma_{\rm E}$ is Euler's
constant. We use this fixed cutoff to avoid an unnecessary parameter.
The cycles exceeding $b$ will enforce the total size, while their
contribution to the maximum will be negligible on the scale $\ell$.

For a cycle profile $c=(c_j)_{j\ge1}$ of total size $n$, enumeration of
permutations gives $n!\prod_j(j^{c_j}c_j!)^{-1}$ permutations.
Consequently the Ewens law of parameter $\theta>0$, conditioned on
$\sum_jc_j=k$, is
\begin{equation}\label{rc:exact-law}
 \Pp_{n,k}(C_j=c_j\text{ for all }j)
   =\frac1{a_{n,k}}\prod_j\frac{j^{-c_j}}{c_j!},
 \qquad \sum_jjc_j=n,\quad\sum_jc_j=k,
\end{equation}
where
\[
 h_N(u)=[z^N](1-z)^{-u}
       =\frac{u(u+1)\cdots(u+N-1)}{N!},
 \qquad a_{N,j}=[u^j]h_N(u).
\]
We set $h_0=1$ and set coefficients with impossible indices to zero.
In particular the conditional law does not depend on $\theta$.
Its normalizing constant is positive for $1\le k\le n$.

Define the long-cycle generating function by
\begin{equation}\label{rc:reservoir-definition}
 \begin{aligned}
 R_b(z,u)&=\exp\left(u\sum_{j>b}\frac{z^j}{j}\right)
       =(1-z)^{-u}e^{-uH_b(z)},\\
 H_b(z)&=\sum_{j=1}^b\frac{z^j}{j},
 \qquad r_{N,j}=[z^Nu^j]R_b(z,u).
 \end{aligned}
\end{equation}
The coefficients $r_{N,j}$ are nonnegative. The next estimate is
absolute in the complex marker $u$; in particular it remains valid
at zeros of $1/\Gamma(u)$.

\begin{lemma}[Scalar reservoir estimate]\label{rc:scalar}
For each fixed $M<\infty$, uniformly for $|u|\le M$, $b\ge2$ and
$N\ge4b$,
\begin{align}\label{rc:scalar-estimate}
 [z^N]R_b(z,u)&=e^{-uH_b}h_N(u)+E_{N,b}(u),\\
 |E_{N,b}(u)|
 &\le C_M\frac bN e^{-(\Ree u)H_b}N^{\Ree u-1}
       +C_M\frac{\log b}{N}e^{-N/(2b)}.\nonumber
\end{align}
\end{lemma}

\begin{proof}
Continue the functions to $|z|<\rho=1+1/b$, slit along $[1,\rho]$,
using the branch of $\log(1-z)$ analytic at zero. First observe that
$H_b(z)+\log(1-z)$ is bounded on $|z|=\rho$, on both banks of the slit.
For $z=\rho e^{i\varphi}$ with $|\varphi|\le2/b$,
\[
 |H_b(z)-H_b|\le Cb|z-1|\le C,\qquad
 \log|1-z|=-\log b+O(1).
\]
For $2/b<|\varphi|\le\pi$, use the radial integral
\[
 H_b(z)+\log(1-z)
   =-\int_0^z\frac{w^b}{1-w}\dd w.
\]
Its modulus is at most
$C|\varphi|^{-1}\int_0^\rho t^b\dd t\le C/(b|\varphi|)$.
The branch arguments are bounded as well. Hence $|R_b(z,u)|\le C_M$
on the outer circle.

Put $F(z,u)=e^{-uH_b}(1-z)^{-u}$ and deform the coefficient contours
for $R_b-F$ to the outer circle, the two banks, and the indentation
$|z-1|=1/N$. For $|w|\le1/b$,
\[
 |H_b(1+w)-H_b|\le Cb|w|,\qquad
 |e^{-u(H_b(1+w)-H_b)}-1|\le C_Mb|w|.
\]
With $a=\Ree u$, the indentation therefore contributes at most
$C_Me^{-aH_b}bN^{a-2}$. The two banks contribute at most
\[
 C_Me^{-aH_b}b
 \int_{1/N}^{1/b} t^{1-a}e^{-Nt/2}\dd t
 \le C_Me^{-aH_b}bN^{a-2}.
\]
Indeed, after $v=Nt$, the remaining integral is bounded by
$\int_1^\infty v^{1-a}e^{-v/2}\dd v$, uniformly for $|a|\le M$.
The positive lower endpoint avoids any issue when $a\ge2$.

It remains to retain the oscillation on the outer circle. Parameterize
it by $z=\rho e^{i\varphi}$, $0<\varphi<2\pi$; its two endpoint values
need not agree. The exact derivative is
\[
 \frac{\dd}{\dd\varphi}R_b(z,u)
      =\frac{iu z^{b+1}}{1-z}R_b(z,u).
\]
Its $L^1$ norm is at most $C_M\log b$, since
$\int_0^{2\pi}|1-\rho e^{i\varphi}|^{-1}\dd\varphi\le C\log b$.
One integration by parts against $e^{-iN\varphi}$, including the
endpoint term, bounds the outer contribution of $R_b$ by
$C_M\rho^{-N}\log b/N$.

For $F$, choose an integer $m>M+2$. Its angular derivative of order
$m$ is a finite sum of bounded-coefficient terms
$e^{-uH_b}z^j(1-z)^{-u-j}$, $1\le j\le m$. Thus
\[
 \int_0^{2\pi}|\partial_\varphi^mF|\dd\varphi
 \le C_M b^{-a}
       \int_0^{2\pi}|1-\rho e^{i\varphi}|^{-a-m}\dd\varphi
 \le C_M b^{m-1}.
\]
Every one-sided endpoint derivative of order $j<m$ has modulus at
most $C_Mb^j$. Integrating by parts $m$ times, with all endpoint
terms retained, bounds the outer contribution of $F$ by
\[
 C_M\rho^{-N}\left\{
       \sum_{j=0}^{m-1}\frac{b^j}{N^{j+1}}
                +\frac{b^{m-1}}{N^m}\right\}
 \le \frac{C_M\rho^{-N}}N.
\]
Finally $\rho^{-N}\le e^{-N/(2b)}$. Combining the three parts of the
contour proves the estimate.
\end{proof}

\begin{lemma}[Marker coefficients and flatness]\label{rc:coefficient}
Put $D=n/T$. Uniformly over the indicated $k$,
\begin{equation}\label{rc:total-coefficient}
 a_{n,k}=\frac{L^k}{n k!\Gamma(\kappa)}
          \{1+O(L^{-1})\}.
\end{equation}
If $0\le d\le D$, $N=n-d$ and $T_N=\log N-H_b$, then, uniformly when
$j/T$ belongs to a fixed positive compact interval,
\begin{equation}\label{rc:marker-coefficient}
 r_{N,j}
 =\frac{T_N^j}{N j!\Gamma(j/T_N)}\{1+O(T^{-1})\}.
\end{equation}
In particular,
\begin{equation}\label{rc:flatness}
 \sup_{0\le d\le D}\left|\frac{r_{n-d,j}}{r_{n,j}}-1\right|
       \le \frac CT.
\end{equation}
\end{lemma}

\begin{proof}
For a function $g$ analytic on a disk of radius greater than
$2\sup j/x$, write $g(u)=\sum_{m\ge0}g_mu^m$. Dividing its marker
coefficient by $x^j/j!$ gives
\[
 \frac{[u^j]e^{xu}g(u)}{x^j/j!}
       =\sum_{m=0}^j g_m\frac{(j)_m}{x^m}
       =g(j/x)+O(x^{-1}),
\]
uniformly when $j/x$ stays in a positive compact interval.
To check the last equality, use Cauchy's geometric bound for $g_m$,
\[
 0\le1-\frac{(j)_m}{j^m}\le\frac{m(m-1)}{2j}
       \quad(0\le m\le j),
\]
and the geometric tail for $m>j$.

Uniform Stirling asymptotics, with the gamma formula interpreted
analytically at nonpositive integers, give
\[
 h_N(u)=N^{u-1}\{1/\Gamma(u)+O_M(N^{-1})\},
       \qquad |u|\le M.
\]
Apply the preceding marker calculation to $g=1/\Gamma$.
For the total coefficient, Cauchy's bound on $|u|=k/L$ makes the
gamma-ratio error relative $O(\sqrt L/n)$; this proves
\eqref{rc:total-coefficient}.

For the reservoir take the marker radius $t=j/T_N$. The first error
in \eqref{rc:scalar-estimate}, after coefficient extraction and
division by $T_N^j/(Nj!\Gamma(t))$, is bounded by
\[
 C\frac bN\,\frac{e^j j!}{j^j}
       \le C\sqrt{T_N}\frac bN.
\]
The outer-circle error is bounded relatively by
$C\sqrt{T_N}\log b\,e^{-N/(2b)}$; here we have harmlessly discarded
an additional factor $e^{-j}$. The gamma-ratio error is
$O(\sqrt{T_N}/N)$. Since $N\ge n-D$, $b\sim n/\ell^2$ and
$T_N\sim T\sim2\log\ell$, all these errors are $O(T^{-1})$.
This proves \eqref{rc:marker-coefficient}.

For flatness keep $j,b$ fixed and differentiate the main term.
With $\psi=\Gamma'/\Gamma$,
\[
 \frac{\dd}{\dd N}
 \log\left\{\frac{T_N^j}{N\Gamma(j/T_N)}\right\}
 =\frac1N\left\{-1+\frac j{T_N}
             +\frac j{T_N^2}\psi(j/T_N)\right\}=O(N^{-1}).
\]
Its relative change is $O(D/n)=O(T^{-1})$.
The two coefficient errors give \eqref{rc:flatness}.
\end{proof}

\begin{theorem}[A single transfer for the entire short configuration]
\label{rc:transfer}
Let $L_*=\sum_{j>b}C_j$ under $\Pp_{n,k}$. Define a reference law
$\Pp_H$ as follows:
\[
 L_*^H\sim\operatorname{Bin}(k,T/L),\qquad q=k-L_*^H,
 \qquad
 \Pp_H(J_v=j\mid L_*^H)=\frac1{jH_b}\quad(1\le j\le b),
\]
where, conditional on $L_*^H$, the $q$ lengths are independent.
Let $C_j^H=\sum_{v=1}^q\ind_{\{J_v=j\}}$.
Then
\begin{equation}\label{rc:joint-tv}
 d_{\TV}\left(
  {\cal L}_{n,k}(L_*,(C_j)_{j\le b}),
  {\cal L}_{H}(L_*^H,(C_j^H)_{j\le b})\right)
 \le C T^{-1/2}\longrightarrow0.
\end{equation}
Here $d_{\TV}(\mu,\nu)=\sup_E|\mu(E)-\nu(E)|$.
The assertion is uniform over $k/L\in K$ and applies to every
measurable event of the whole short configuration.
\end{theorem}

\begin{proof}
For a short profile $c$ put $q=\sum_{j\le b}c_j$ and
$d=\sum_{j\le b}jc_j$. Conditional on $L_*^H=l$, where $q=k-l$,
its harmonic probability is
\[
 \frac{q!}{H_b^q}\prod_{j\le b}\frac{j^{-c_j}}{c_j!}.
\]
Combining this with the binomial mass shows that the reference joint
probability of $(l,c)$ is
\[
 \frac{k!T^l}{l!L^k}\prod_{j\le b}\frac{j^{-c_j}}{c_j!}.
\]
Its exact counterpart is
\[
 \frac{r_{n-d,l}}{a_{n,k}}
          \prod_{j\le b}\frac{j^{-c_j}}{c_j!}.
\]
We define $r_{n-d,l}=0$ when $d>n$. Thus the exact law has density
\begin{equation}\label{rc:joint-density}
 f(l,c)=\frac{L^kl!}{k!T^la_{n,k}}r_{n-d,l}
\end{equation}
with respect to the reference law, including profiles with impossible
total size.

Use the joint good set
\[
 G=\{\kappa_-T/2\le l\le2\kappa_+T,\ d\le D\}.
\]
On $G$, \eqref{rc:total-coefficient} and
\eqref{rc:marker-coefficient} give
\[
 f(l,c)=A_l R_l(d),\qquad
 A_l=\frac{\Gamma(\kappa)}{\Gamma(l/T)}\{1+O(T^{-1})\},
 \qquad R_l(d)=\frac{r_{n-d,l}}{r_{n,l}}.
\]
The constants are uniform on the displayed count interval. The gamma
ratio is bounded there, and flatness gives $R_l(d)=1+O(T^{-1})$ on
$G$. Thus
\begin{equation}\label{rc:good-density}
 |f(l,c)-1|\le C|l/T-\kappa|+C/T\qquad\text{on }G.
\end{equation}

Only a first moment of the harmonic degree is needed. Conditional on
$l$, the $q=k-l$ lengths have mean $b/H_b$, so
\begin{equation}\label{rc:degree-moment}
 \E_H(d\mid l)=\frac{qb}{H_b}\le Cb.
\end{equation}
Markov's inequality gives
$\Pp_H(d>D)\le Cb/D=O(T/\ell^2)=o(T^{-1})$.

Under the reference law,
$\E_H L_*^H=\kappa T$ and $\Var_H(L_*^H)\le\kappa_+T$.
Chebyshev's inequality bounds the probability that $l$ leaves its
displayed interval by $C/T$. Combining it with the degree estimate
gives $\Pp_H(G^c)\le C/T$, and
\[
 \E_H|L_*^H/T-\kappa|\le C/\sqrt T.
\]
It follows that
$\E_H[|f-1|\ind_G]\le C/\sqrt T$.
No coefficient estimate is required anywhere on $G^c$. Because $f$
is the density of a probability measure, $\E_H f=1$, and hence
\[
 d_{\TV}=\tfrac12\E_H|f-1|
        =\E_H(1-f)_+
        \le\E_H[|f-1|\ind_G]+\Pp_H(G^c)
        \le C/\sqrt T.
\]
This proves the theorem. The normalization of the full density is
what allows both bad count values and large short degrees to be
discarded using their reference probabilities alone.
\end{proof}

The transfer is used only after the harmonic argument has proved a
high-probability statement about the maximum. It does not give relative
accuracy for arbitrary rare events. In particular, the precise
one-point and two-point estimates needed inside the harmonic
second-moment argument must be established there before this theorem
is applied. No repeated coefficient comparison through those path
constraints is necessary. Adding independent auxiliary randomness
to both laws preserves the same total variation bound.

\begin{lemma}[Deterministic insertion]\label{rc:insertion}
Let $P$ be a nonzero polynomial of degree $d\ge1$ and
$Q(z)=\prod_{v=1}^l(1-z^{j_v})$, with positive integers $j_v$. Then
\[
 -\log2-2ed\sum_{v=1}^l j_v^{-1}
 \le\log\|PQ\|_\infty-\log\|P\|_\infty
 \le l\log2.
\]
For constant $P$, the lower bound is zero.
\end{lemma}

\begin{proof}
The upper bound is pointwise. For completeness,
$\sup_{|z|\le1}|P'(z)|\le ed\|P\|_\infty$: outside the unit disk
the maximum principle for the reversed polynomial gives
$|P(z)|\le|z|^d\|P\|_\infty$, and Cauchy's derivative estimate on
a circle of radius $1/d$ proves the claim.
At a maximizing point $z_*$ on the unit circle set
$\rho=1-(2ed)^{-1}$. Then
$|P(\rho z_*)|\ge\|P\|_\infty/2$. Also
\[
 |Q(\rho z_*)|\ge\prod_v(1-\rho^{j_v}),\qquad
 -\log(1-e^{-x})\le(e^x-1)^{-1}\le x^{-1}\quad(x>0).
\]
Since $\rho^j\le e^{-j/(2ed)}$, the lower bound follows by evaluating
$PQ$ at $\rho z_*$ and applying the maximum modulus principle.
For constant $P$ use $Q(0)=1$.
\end{proof}

\begin{proposition}[Restoration and centering]\label{rc:restoration}
Let
\[
 \begin{gathered}
 M_b=\log\left\|\prod_{j\le b}(1-z^j)^{C_j}\right\|_\infty,
 \qquad q=k-L_*,\qquad a=\log b,\\
 m^H_{b,q}=\inf_{s>0}\frac{a+q\lambda(s)}s.
 \end{gathered}
\]
Under $\Pp_{n,k}$, uniformly for $k/L\in K$,
\begin{align}\label{rc:restore-bound}
 |M_n-M_b|&=O_{\Pp}(T)=o_{\Pp}(\ell),\\
 0\le m_{n,k}-m^H_{b,q}
       &=O_{\Pp}(T)=o_{\Pp}(\ell).\label{rc:center-change}
\end{align}
Here uniform $O_{\Pp}$ means uniform tightness after division by the
displayed scale.
For $q>0$, this harmonic center equals $a v(q/a)$; for $q=0$ it
is zero. With probability tending uniformly to one, $\kappa'=q/a$
belongs to a fixed positive compact interval. Every continuous
$c:(0,\infty)\to\R$ satisfies
$c(\kappa')\log a=c(\kappa)\ell+o_{\Pp}(\ell)$, where the left side
may be defined arbitrarily on the exceptional event $q=0$.
\end{proposition}

\begin{proof}
The first-moment bound \eqref{rc:degree-moment} gives
$d/b=O_{\Pp}(1)$ under the harmonic law for
$d=\sum_{j\le b}jC_j$, and the binomial law gives $L_*=O_{\Pp}(T)$.
Both tightness statements transfer by \eqref{rc:joint-tv}. Since
every omitted length exceeds $b$, Lemma~\ref{rc:insertion} bounds
the lower loss by
$\log2+2edL_* /(b+1)=O_{\Pp}(T)$.
The upper loss is at most $L_*\log2$. This proves
\eqref{rc:restore-bound} without a conditional cross-moment estimate.

For the centers, $L-a=T+O(1)$ and $L_*=O_{\Pp}(T)$ imply
\[
 \kappa'-\kappa=\frac{\kappa(L-a)-L_*}{a}
                  =O_{\Pp}(T/L)=o_{\Pp}(1).
\]
Thus $\kappa'$ belongs to a slightly enlarged positive compact
interval with probability tending uniformly to one. On this event,
put $s'=s_{\kappa'}$. Since $\lambda(s)>0$ for $s>0$,
increasing either coefficient in the variational objective increases
its infimum. Evaluating the larger objective at the smaller one's
minimizer gives the deterministic bounds
\[
 0\le m_{n,k}-m^H_{b,q}
 \le\frac{L-a+L_*\lambda(s')}{s'}
 \le C(T+L_*).
\]
The constant is uniform because $s'$ ranges over a positive compact
interval. This proves \eqref{rc:center-change} directly, without a
Taylor expansion or a sharper estimate than the target scale needs.
Here positivity of $\lambda$ follows from
$\lambda(0)=\lambda'(0)=0$ and $\lambda''>0$.
Uniform continuity of $c$ on the enlarged compact interval, together
with $\log a=\ell+o(1)$, proves the correction-scale assertion.
\end{proof}

%% file: reference.tex
\section{Uniform reference excursions and exact count conditioning}
\label{sec:reference}

This section concerns independent tilted marks in an auxiliary
exponential-gap environment.  It does not assume a branching structure.
All constants and all convergence statements are uniform when
$\kappa$ and $s$ range over prescribed compact subintervals of
$(0,\infty)$.  We initially allow these two parameters to vary
independently.  Write
\[
 \mu_s=\lambda'(s),\qquad
 \sigma_s^2=s^2\lambda''(s),\qquad
 \beta_{\kappa,s}=(\kappa\sigma_s)^{-1},\qquad
 \alpha_{\kappa,s}=2\gamma(\beta_{\kappa,s})+\tfrac12.
\]
Here $\gamma$ is the Brownian random-wall persistence exponent.
Let $V_1,V_2,\ldots$ have the tilted log-sine law, let
$E_1,E_2,\ldots$ be exponential variables of rate $\kappa$, and make
the two sequences independent.  Our reference coordinates are
\begin{equation}\label{rr:coordinates}
 Y_i=s(V_i-\mu_s),\qquad A_i=E_i-\kappa^{-1},
 \qquad S_j=\sum_{i=1}^j(Y_i-A_i).
\end{equation}
If $s=s_\kappa$, the critical equation gives
$Y_i-A_i=sV_i-\lambda(s)-E_i$.
Probabilities over the marks at a fixed gap environment are denoted
by $\mathbb P_E$.  Uniformity below means a supremum of
\emph{environmental failure probabilities}; it does not assert an
almost-sure coupling of all parameter values.

\subsection{The two probabilistic inputs}

We use the following form of Brownian persistence
\cite[Theorem~1.5]{MalleinMilosWall}.  For independent standard
Brownian motions $B,W$, a positive function $g$ with
$\inf g>0$ and $\log g(t)=o(\log t)$, and a deterministic $f$ satisfying
$f(0)=0$ and $f(t)=o(t^{1/2-\varepsilon_0})$ for some
$\varepsilon_0>0$,
\begin{equation}\label{rr:brownian-input}
\begin{split}
 -\frac1{\log T}\log\mathbb P\bigl\{&
       g(T)+B_t\geq\beta W_t+f(t),\quad 0\leq t\leq T,\\
       &B_T-\beta W_T\in[u\sqrt T,v\sqrt T]\mid W\bigr\}
 \longrightarrow\gamma(\beta)
\end{split}
\end{equation}
almost surely, for $0\leq u<v\leq\infty$.
Open or closed terminal endpoints give the same probability.
The function $\gamma$ is finite and convex, hence continuous.

The second input is Sakhanenko's inequality in the explicit form
stated in \cite[Theorem~1]{ZaitsevStrong}.
If independent centered variables $Z_i$ satisfy, for one $\tau>0$,
\[
 \mathbb E|Z_i|^3e^{|Z_i|/\tau}\leq\tau\mathbb EZ_i^2,
\]
they admit a Gaussian coupling with the same variances such that
\begin{equation}\label{rr:sakhanenko}
 \mathbb E\exp\left\{\frac{c}{\tau}
       \max_{j\leq m}\left|\sum_{i\leq j}(Z_i-G_i)\right|\right\}
 \leq 1+\frac{\sqrt{\sum_{i\leq m}\mathbb EZ_i^2}}{\tau}.
\end{equation}
The constant $c>0$ is absolute.

\begin{lemma}[Uniform one-sided probabilities]\label{rr:onesided}
Fix $0<\nu<1/2$, $0<a_-<a_+<\infty$, $C<\infty$, and a compact
set of pairs $(u,v)$ with $0<u<v<\infty$.
For every $\varepsilon>0$, with environmental probability tending
uniformly to one, simultaneously over
$a\in[a_-,a_+]$, $c_0\in[0,C]$, and these pairs $(u,v)$,
\[
 \mathbb P_E\{S_j\leq a\log m-c_0j^\nu,\ j\leq m,\
               S_m\in[-v\sqrt m,-u\sqrt m]\}
 \geq m^{-\gamma(\beta_{\kappa,s})-\varepsilon}.
\]
The assertion also holds for the reflected increments $-Y_i+A_i$.
For every fixed $P>0$ and $0<C_1<\infty$,
\[
 \mathbb P_E\{S_j\leq C_1(\log m)^P,\ j\leq m\}
 \leq m^{-\gamma(\beta_{\kappa,s})+\varepsilon}
\]
with environmental probability tending uniformly to one; replacing
the displayed buffer by a larger fixed positive logarithmic buffer
when $P<1$ is harmless.
\end{lemma}

\begin{proof}
We give the uniformity and small-buffer details, since a coupling
error of order $\log m$ cannot simply be subtracted from an
arbitrarily small multiple of $\log m$.

The laws of $Y_i,A_i$ and their reflections have uniformly bounded
exponential moments in a common neighborhood of zero, and their
variances are bounded away from zero.  For the marks this follows
also from
\[
 \mathbb Ee^{tY_i}
   =e^{-ts\mu_s+\lambda(s+st)-\lambda(s)}.
\]
Consequently they all satisfy the hypothesis of
\eqref{rr:sakhanenko} with one $\tau$: first bound the exponential
third moment in a common neighborhood, and then enlarge $\tau$
using the variance lower bound.  For any fixed $J$, we therefore
have couplings through $m$ steps with error at most $C_J\log(m+1)$
and failure probability at most $m^{-J}$.  Couple marks and gaps
independently.  Gaussian interpolation between integer times costs
another $O(\log m)$, with equally small failure, by the Gaussian
tail bound and a union bound.  In particular the mark-coupling
failure is a deterministic conditional error when the environment
and its coupled Brownian wall have been fixed.

We next explain uniform Brownian persistence.  For fixed $W$, the
negative logarithm of the probability in
\eqref{rr:brownian-input}, divided by $\log T$, is convex jointly in
$(\beta,a,c_0,u,v)$ when
$g(T)=a\log(e+T)$ and $f(t)=c_0t^\nu$.
On finite meshes this follows from Pr\'ekopa's inequality applied
to a Gaussian density and affine inequalities, as in the Gaussian
convexity argument of \cite[Lemma~2.6]{MalleinMilosWall}.  Decreasing dense-mesh limits give the
continuous-path statement.  Apply \eqref{rr:brownian-input} on a
countable dense set of parameter tuples.  On each fixed compact
subset the probabilities are positive for all sufficiently large $T$:
the terminal interval is of order $\sqrt T$, whereas the curve is
$o(\sqrt T)$, and the initial buffer is positive.
Convergence of these nonnegative convex functions is locally uniform on
$\{a>0,\ 0<u<v<\infty\}$.  Indeed, convergence at the vertices of
a slightly larger polytope bounds the functions above there;
nonnegativity and convexity then give a common Lipschitz bound on
an inner compact set.  Dense-point convergence and continuity of
$\gamma$ prove the assertion.  The same argument applies to
$g(T)=a(\log(e+T))^P$ for $P\geq1$, and separately to the terminal half-line
$[0,\infty)$, which is explicitly allowed in
\eqref{rr:brownian-input}.

For unrestricted Brownian survival, Harris association gives
\[
 \mathbb P(\mathcal U\mid W)
 \leq \frac{\mathbb P(\mathcal U,\
                    B_T-\beta W_T\geq0\mid W)}
                 {\mathbb P(B_T-\beta W_T\geq0\mid W)}.
\]
Harris' inequality applies to the independent increments on each
finite mesh. Passing to dense meshes gives the continuous version.
On $|W_T|\leq K\sqrt T$ the denominator is bounded below uniformly
for compact $\beta$.  Its complement has arbitrarily small
probability as $K\to\infty$.  This proves the uniform Brownian
upper exponent.  The coupling, with an enlarged buffer, proves
the claimed random-walk upper bound.  It also proves the lower
bound when the logarithmic buffer coefficient exceeds the
uniform coupling constant, using a strictly smaller terminal
interval and $J>\sup\gamma+2$.

For a smaller positive buffer, set
$m_0=\lfloor(\log m)^2\rfloor$.  Fix a sufficiently large $K$.
We claim that
\begin{equation}\label{rr:startup}
 \mathbb P_E\{S_j\leq(a_-/2)\log m,\ j\leq m_0,\
               S_{m_0}\in[-2K\log m,-K\log m]\}
 \geq m_0^{-J}
\end{equation}
with environmental probability tending uniformly to one, for a
fixed $J$.  Here is a precise uniform positivity argument.
Put $z=\sigma_s^{-1}$ and normalize Brownian time to $[0,1]$.
For a fixed continuous wall $W$, consider
\[
 q(\beta,z;W)=
 \mathbb P\{B_t-\beta W_t\leq a_-z/4,\ t\leq1,\
       B_1-\beta W_1\in[-7Kz/4,-5Kz/4]\mid W\}.
\]
This probability is positive: a sufficiently small uniform tube
around $\beta W_t-(3Kz/2)t$ is contained in the event, and Wiener
measure has full support on continuous paths starting at zero.
For this elementary support fact, approximate the path by a polygonal
path, require its finitely many Gaussian increments to lie in small
intervals, and require each intervening Brownian bridge to stay in a
small uniform tube; each of these finitely many probabilities is positive.
After this normalization, all inequalities are affine in
$(\beta,z)$, so $q$ is jointly log-concave.  On a rectangle
containing the compact parameter family,
\[
 q(\beta,z;W)\geq
  \min_{\text{four corners }(\beta_0,z_0)}q(\beta_0,z_0;W)
  =:Q(W)>0.
\]
Thus $\mathbb P\{Q(W)\leq2m_0^{-J}\}\to0$, without any lower-tail
rate assumption.  The local coupling error is
$O(\log m_0)=O(\log\log m)=o(\log m)$; the Brownian event has
strict barrier and endpoint slack.  Subtracting a mark-coupling
failure of at most $m_0^{-J}$ proves \eqref{rr:startup}.
The Brownian wall has its standard marginal at each $m$;
no almost-sure identification of the walls for different $m$
is required.

The curved loss through startup is
$Cm_0^\nu=o(\log m)$.  From any startup endpoint
$x\in[-2K\log m,-K\log m]$, require the remaining walk to stay below
\[
 (K-1)\log m-Cj^\nu.
\]
This is a large-buffer event.  Since
$(m_0+j)^\nu\leq m_0^\nu+j^\nu$, it lies below the desired
continuation after adding $x$.  A terminal interval strictly inside
$[-v\sqrt m,-u\sqrt m]$ works for every such $x$, since
$x=O(\log m)=o(\sqrt m)$ and the permitted interval widths have a
common positive lower bound.  Its uniform probability is at least
$m^{-\gamma-\varepsilon/2}$.  Conditional independence of the mark
segments multiplies this by $m_0^{-J}=m^{-o(1)}$.
This proves the lower estimate.  Reflections have the same
coupling constants and variance ratio, proving the last assertion.
\end{proof}

\subsection{Gap conditioning and a pointwise joining estimate}

\begin{lemma}[A central reserve for exact conditioning]\label{rr:gaps}
For sufficiently large $n_0$, condition $n_0$ independent exponential gaps of rate $\kappa$ to
have total
\[
 T=n_0/\kappa+\Delta,\qquad |\Delta|\leq D\log n_0.
\]
For every fixed $\delta>0$, the joint law of any selected
$m\leq(1-\delta)n_0$ gaps has density at most $C_\delta$
relative to their independent law.  The bound is uniform in
the selected values, $\Delta$, and the parameters.
Moreover
\[
 \max_{j\leq n_0}
 \left|\sum_{i\leq j}E_i-jT/n_0\right|=O_{\mathbb P}(\sqrt{n_0}),
 \qquad
 \mathbb P\{\max_iE_i>C\log n_0\}\longrightarrow0
\]
for a sufficiently large fixed $C$, uniformly.
Each individual gap is uniformly tight.
There also exist fixed $C_0,C_1,c_1>0$ such that, for $u\geq0$,
\begin{equation}\label{rr:gap-envelope}
\begin{split}
 &\mathbb P\left\{
 \max_{0\leq j\leq n_0}
 \max\left(\sum_{i=1}^j E_i-C_0j,
           \sum_{i=n_0-j+1}^{n_0}E_i-C_0j\right)>u
 \right\}\\
 &\hspace{25mm}\leq C_1e^{-c_1u}+C_1e^{-c_1n_0}.
\end{split}
\end{equation}
Empty sums are zero.  These constants are uniform in the
same parameters.
\end{lemma}

\begin{proof}
Writing
$g_j(x)=\kappa^jx^{j-1}e^{-\kappa x}/\Gamma(j)$ for $x>0$,
the selected density is
\[
 \frac{g_{n_0-m}(T-\sum_{\mathrm{selected}}E_i)}{g_{n_0}(T)}.
\]
The numerator is zero off its support.  Stirling's formula gives
$g_{n_0}(T)\asymp\kappa/\sqrt{n_0}$ uniformly in the stated
range, while
$\sup_xg_{n_0-m}(x)\leq C\kappa/\sqrt{n_0-m}$.
This proves the density bound.

For independent rate-one exponentials $Z_i$, the conditioned
vector has the representation
\[
 (E_i)_{i\leq n_0}
   \ \stackrel{\mathrm{law}}=\ 
 \left(\frac{T Z_i}{\sum_{h=1}^{n_0}Z_h}\right)_{i\leq n_0}.
\]
The maximal inequality applied to the centered partial sums
of $Z_i$ proves the clock-bridge bound.
To prove the linear envelope, take a uniform bound $T/n_0\leq B$.
Outside the event $\sum_i Z_i<n_0/2$, whose probability is at
most $(\sqrt e/2)^{n_0}$, the multiplier $T/\sum_iZ_i$ is at
most $2B$.  For $v\geq0$, an exponential Markov inequality and
a union bound give
\[
 \mathbb P\left\{\max_{j\leq n_0}
       \left(\sum_{i=1}^j Z_i-2j\right)>v\right\}
 \leq e^{-v/2}\sum_{j\geq1}(2/e)^j.
\]
The same bound holds for reversed partial sums.  Taking
$C_0=4B$ and $v=u/(2B)$ proves \eqref{rr:gap-envelope}.
Finally a Dirichlet marginal gives
\[
 \mathbb P(E_i>x)=(1-x/T)^{n_0-1}
 \leq e^{-(n_0-1)x/T},\qquad 0\leq x\leq T.
\]
A union bound proves the maximum-gap assertion and the same
inequality proves individual tightness.
\end{proof}

\begin{lemma}[Pointwise mark bridge density]\label{rr:density}
Let $Z_i=s(V_i-\mu_s)$.  For large $m$, the density of
$\sum_{i\leq m}Z_i$ is bounded above by $C/\sqrt m$ at every
point.  For each fixed $K<\infty$ there exist $m_0(K)$,
$W<\infty$ and $c_K>0$ such that, for every $m\geq m_0(K)$
and simultaneously for $|d|\leq K\sqrt m$,
the measure
\[
 \mathbb P\left\{\sum_{i\leq m}Z_i\in\mathrm d d,\
       \max_{j\leq m}
       \left|\sum_{i\leq j}Z_i-\frac jm d\right|
       \leq W\sqrt m\right\}
\]
has density at least $c_K/\sqrt m$ at $d$.
All constants are uniform for $s$ in its compact interval.
The density version is the one obtained by integrating the product
increment densities against the prescribed terminal sum.
\end{lemma}

\begin{proof}
The tilted mark has density
\[
 f_s(v)=
 \frac{e^{(s+1)v-\lambda(s)}}
      {\pi\sqrt{1-e^{2v}/4}}\,\mathbf1_{\{v<\log2\}}.
\]
Its characteristic function is $O((1+|t|)^{-1/2})$ uniformly:
split at distance $|t|^{-1}$ from the square-root endpoint
singularity and integrate by parts on the complement; the
negative tail and its derivative are exponentially integrable.
The same decay holds for the characteristic function $\phi_s$
of $Z_i$, since $s$ is bounded above and away from zero.
Write $\sigma_s^2=\operatorname{Var}(Z_i)$ and let $p_{m,s}$
be the convolution density of $R_m=\sum_{i\leq m}Z_i$.
We first check that this canonical convolution version agrees
pointwise with Fourier inversion once $m\geq3$.
If $g_s$ is the density of $Z_i$ and
$z_s=s(\log2-\mu_s)$, the displayed formula gives, uniformly,
\[
 g_s(z_s-x)\leq Cx^{-1/2}e^{-cx}\quad(x>0),
 \qquad g_s(z)=0\quad(z\geq z_s).
\]
Consequently, for $u>0$,
\[
 p_{2,s}(2z_s-u)
 \leq Ce^{-cu}\int_0^u\frac{\dd x}{\sqrt{x(u-x)}}
 =C\pi e^{-cu};
\]
the density is zero for $u\leq0$.
Thus $p_{2,s}$ is uniformly bounded.  For $m\geq3$, Tonelli's
theorem identifies the product-density integral at each terminal
value with $p_{2,s}*p_{m-2,s}$.  This convolution is continuous:
its change under translation by $h$ is at most
$\|p_{2,s}\|_\infty
  \|p_{m-2,s}(\cdot+h)-p_{m-2,s}\|_1$, which tends to zero.
The characteristic function $\phi_s^m$ is integrable for
$m\geq3$ by the preceding decay estimate.  Its continuous
inverse agrees almost everywhere with $p_{m,s}$, and therefore
everywhere.  This justifies using the stated pointwise density
version throughout the following estimates.

The variances are bounded above and away from zero, and the
third absolute moments are uniformly bounded.  Thus
\[
 \phi_s(t)=1-\tfrac12\sigma_s^2t^2+O(|t|^3),
 \qquad |\phi_s(t)|\leq e^{-ct^2}
\]
near zero, uniformly in $s$.  On any compact annulus away
from zero, $|\phi_s|$ is uniformly less than one, by continuity
in $s$ and the fact that each law has a density.
These estimates imply
\[
 \sup_s\int_{\mathbb R}
 \left|\phi_s(t/\sqrt m)^m-e^{-\sigma_s^2t^2/2}\right|
 \,\mathrm dt\longrightarrow0.
\]
Indeed, Taylor expansion and the Gaussian bound handle
$|t|\leq\delta\sqrt m$ by dominated convergence.  The annulus
$\delta\sqrt m\leq|t|\leq R\sqrt m$ contributes
$O(\sqrt m\rho^m)$ for some $\rho<1$.  On its complement,
the bound $|\phi_s(u)|\leq C|u|^{-1/2}$ makes the integral
exponentially small after choosing $R$ sufficiently large.
Fourier inversion therefore gives the uniform local central
limit estimate
\[
 \sup_s\sup_{x\in\mathbb R}
 \left|\sqrt m\,p_{m,s}(\sqrt m x)
 -\frac{e^{-x^2/(2\sigma_s^2)}}{\sqrt{2\pi\sigma_s^2}}
 \right|\longrightarrow0.
\]
In particular, for large $m$,
\[
 \sup_{s,d}p_{m,s}(d)\leq\frac C{\sqrt m},
 \qquad
 \inf_{s,\,|d|\leq K\sqrt m}p_{m,s}(d)
       \geq\frac{a_K}{\sqrt m}
\]
for some $a_K>0$.

Fix $|d|\leq K\sqrt m$, put $m_1=\lfloor m/2\rfloor$
and $m_2=m-m_1$, and take $W>K$.
For the first half, define
\[
 F_d=\left\{\max_{j\leq m_1}
           |R_j-jd/m|>W\sqrt m\right\}.
\]
Since $|jd/m|\leq K\sqrt m$, Kolmogorov's maximal inequality
gives $\mathbb P(F_d)\leq C/(W-K)^2$.
Independence of the remaining increments bounds the density
contributed by this failure event, at the prescribed total $d$, by
\[
 \mathbb E\bigl[\mathbf1_{F_d}
                   p_{m_2,s}(d-R_{m_1})\bigr]
 \leq\frac{C'}{(W-K)^2\sqrt m}.
\]
For $j\geq m_1$, at total $R_m=d$ the identity
\[
 R_j-jd/m
 =-\left((R_m-R_j)-(m-j)d/m\right)
\]
turns a second-half failure into the same event for the reversed
first $m_2$ increments.  Its density contribution has the
same bound, now using $p_{m_1,s}$ for the unexamined increments.
This also covers odd $m$.
Choose $W$ so that $2C'/(W-K)^2\leq a_K/2$.
Subtracting both failure contributions from the unrestricted
density leaves at least $a_K/(2\sqrt m)$ inside the tube.
All these inequalities concern the product-density integrals
at the prescribed total, so they hold pointwise in $d$ for
the version stated in the lemma.
\end{proof}

\subsection{One master excursion estimate}

\begin{theorem}[Uniform conditioned excursions]\label{rr:excursion}
Let $N$ marks and $N+1$ gaps be as in \eqref{rr:coordinates},
and condition the gaps on
\[
 T=\sum_{i=1}^{N+1}E_i=(N+1)/\kappa+\Delta,
 \qquad |\Delta|\leq D\log N.
\]
Fix $0<\nu<1/2$, $0<a_-<a_+<\infty$, $A,D_1<\infty$, and
$0<w_-<w_+<\infty$.  For
\[
 h,h'\in[a_-\log N,a_+\log N],\quad
 |d|\leq D_1\log N,\quad 0\leq a\leq A,\quad
 w\in[w_-,w_+],
\]
set
\[
 b_j=h-dj/N-a\min(j,N-j)^\nu,\qquad z=h-d-h'.
\]
For every $\varepsilon>0$, with conditioned environmental
probability tending uniformly to one, simultaneously over these
choices,
\begin{equation}\label{rr:excursion-bounds}
 N^{-\alpha_{\kappa,s}-\varepsilon}
 \leq\mathbb P_E\{S_j\leq b_j,\ j\leq N,\
                         S_N\in[z,z+w]\}
 \leq N^{-\alpha_{\kappa,s}+\varepsilon}.
\end{equation}
The upper bound also holds for any fixed $O(\log N)$ constant
barrier and any $O(\log N)$ terminal location, without positive
initial or final gaps.
\end{theorem}

\begin{proof}
Put $p=\lfloor N/3\rfloor$.  For the upper bound retain only
the first and reversed last $p$ path constraints, enlarging both
barriers to $C\log N$.  The reversed, negated last increments
have the form $-Y_i+A_i$.  Conditional on the gaps, the three
mark groups are independent.  Lemma~\ref{rr:density} bounds
the middle endpoint density at every location by $C/\sqrt N$.
Lemma~\ref{rr:onesided} bounds each end survival by
$N^{-\gamma+\varepsilon/4}$ on an iid environmental event
of probability tending uniformly to one.
Both end events use only the first and last $p$ gaps.
Lemma~\ref{rr:gaps}, leaving a central reserve, transfers their
joint typicality to the conditioned environment.  Multiplying
these three bounds and the terminal width proves the upper
estimate.  None of these enlarged end events depends on the
moving terminal location.

For the lower bound, require each end process to stay below
\[
 (a_-/2)\log N-(A+1)j^\nu,\qquad j\leq p,
\]
and to finish in
$[-(H+1)\sqrt N,-H\sqrt N]$, where $H$ is chosen below.
Lemma~\ref{rr:onesided} gives each end probability
$N^{-\gamma-\varepsilon/4}$ on typical iid environments.
These stricter paths work for every displayed parameter:
uniformly for $j\leq p$,
\[
 D_1(\log N)j/N=o(j^\nu).
\]
For an actual terminal value $t=z+u$, $0\leq u\leq w_+$,
the allowed reversed barrier is
$h'+dj/N-aj^\nu-u$, which contains the same stricter end event.
Lemma~\ref{rr:gaps} again transfers typicality of the two ends.

Except on an arbitrarily small environmental set, the centered
clock bridge is bounded by $K\sqrt N$, uniformly.
If the first endpoint is $x$ and the reversed last endpoint is
$r$, the middle must join $x$ to $t+r$.
Both endpoints lie at depth approximately $H\sqrt N$.
The required centered mark total is $O_K(\sqrt N)$;
the $H\sqrt N$ terms cancel in its difference.
Lemma~\ref{rr:density} supplies a pointwise joining density
$c_K/\sqrt N$ in a fixed-width $\sqrt N$ tube.
The centered gap chord error is at most $C K\sqrt N$.
Choose $H$ larger than this constant and the mark tube width.
The joined path then stays at depth at least $c\sqrt N$,
whereas every central barrier is bounded below by
$-D_1\log N-A N^\nu=-o(\sqrt N)$.
It is therefore admissible.

Integrating the two end events and the terminal interval of
width at least $w_-$ gives the lower bound, after absorbing the
fixed $c_K$ into $N^{-\varepsilon/2}$.
For each desired environmental error choose $K$, then $H$, first.
The resulting limiting failure probability is at most that
arbitrarily small error, proving convergence to zero.
All choices are uniform on the parameter compact set, and the
same end events and joining estimate cover all displayed
$h,h',d,a,w$ simultaneously.
\end{proof}

\begin{lemma}[Simultaneous early and late deletions]\label{rr:shifts}
In the conditioned environment of Theorem~\ref{rr:excursion},
fix $0<C,P<\infty$ and put
$t_N=(\log N)^C$, $u_N=(\log N)^P$, increasing $P$ if necessary.
For every $\varepsilon>0$, with probability tending uniformly
to one, simultaneously for integers $0\leq k,k'\leq t_N$,
$|b|,|z|\leq u_N$ and $0<w\leq u_N$,
\[
 \mathbb P_E\{S_{k+j}-S_k\leq b,\ j\leq N-k-k',\
       S_{N-k'}-S_k\in[z,z+w]\}
 \leq N^{-\alpha_{\kappa,s}+\varepsilon}.
\]
\end{lemma}

\begin{proof}
Set $p=\lfloor N/3\rfloor$ and $b_N=\log(t_N+2)$.
Before conditioning the gaps, write
\[
 A_k=\mathbb P_E\{S_{k+j}-S_k\leq u_N,\ j\leq p-k\},
 \qquad C_j=\mathbb P_E\{S_i\leq b_N,\ i\leq j\}.
\]
Concatenation and independence of disjoint mark groups give
\[
 A_k\leq
 \frac{\mathbb P_E\{S_j\leq u_N+b_N,\ j\leq p\}}
      {C_{\lceil t_N\rceil}}.
\]
This single ratio bounds every $k$ in the range.
Lemma~\ref{rr:onesided} bounds the numerator by
$N^{-\gamma+o(1)}$ in probability, uniformly, and bounds the
denominator below by
$t_N^{-\gamma-o(1)}=N^{-o(1)}$.
Apply the identical argument to the reversed, negated last
increments, with $2u_N$ in place of $u_N$.
These ratio events depend only on the first and last $p$ gaps,
so Lemma~\ref{rr:gaps} transfers their joint typicality.

For any event in the statement the middle mark group is exactly
the group with indices $p+1,\ldots,N-p$, independently of $k,k'$.
Its density is at most $C/\sqrt N$.  The reversed last barrier
is at most $b-z\leq2u_N$.  Multiplication of the two simultaneous
end bounds and $Cw/\sqrt N$ proves the assertion, since
$u_N=N^{o(1)}$.
\end{proof}

A useful consequence is the weighted suffix bound.
Suppose a reference suffix kernel $K_k(x)$ lies in $[0,1]$,
has a terminal exponential weight bounded by one, and, for
$B_k-H_N\leq x\leq B_k$, its centered barriers and terminal
window fall in the polynomial-logarithmic ranges of
Lemma~\ref{rr:shifts}.  If the omitted initial number of marks
is polynomial-logarithmic, $\theta\geq\theta_0>0$, and
$H_N=(\log N)^P$ with $P>1$ sufficiently large, then
\begin{equation}\label{rr:weighted-suffix}
 \sup_{x\leq B_k}e^{\theta s(x-B_k)}K_k(x)
 \leq N^{-\alpha_{\kappa,s}+\varepsilon}
\end{equation}
on the same type of typical environment.
For deficits at most $H_N$ use Lemma~\ref{rr:shifts};
larger deficits contribute at most
$e^{-\theta_0s_{\min}H_N}$.
The lemma likewise bounds a late prefix whose endpoint is
restricted to a polynomial-logarithmic interval.

\subsection{Physical clocks, fixed counts, and bounded noise}

\begin{lemma}[Transfer to fixed block counts]\label{rr:counts}
Assume now $s=s_\kappa$.
Take $N$ independent uniform clock points in $[0,T]$, where
$|T-(N+1)/\kappa|\leq D\log N$, and an arbitrary deterministic
partition $0=t_0<\cdots<t_m=T$.
Their block counts are random; conditional on them, give the
marks in each block their independent tilted laws.  Write $Q_i=Q(t_i)$
and $S_i^{\rm c}=s\sum_{j\leq Q_i}V_j-\lambda Q_i-t_i$.
The following assertions hold.

\begin{enumerate}
\item
Suppose every block width is at most $C\log N$ and there is a
regular-count event $\mathcal R_N$, of probability tending uniformly
to one, on which every block count is at most $C\log N$.
A coarse event $S_i^{\rm c}\leq C'\log N$ with terminal interval
at an $O(\log N)$ location and fixed bounded width has the upper
exponent in Theorem~\ref{rr:excursion}, on a set of count vectors
whose probability tends uniformly to one.
Without assuming $\mathbb P(\mathcal R_N)\to1$, the conclusion is
that the probability of an upper-bound violation intersected with
$\mathcal R_N$ tends to zero.
The same statement holds simultaneously for the polynomial-logarithmic
start/end shifts and early/late mark deletions of
Lemma~\ref{rr:shifts}.
\item
Let the full height barrier at physical clock time $t$ be
\[
\begin{aligned}
 B(t)&=C_0+\frac{t+\lambda(s)Q(t)}s-G(t),\\
 G(t)&=G(0)+d_0t/T+
 a_0\bigl[(\min(t,T-t)+\tau)^\nu-\tau^\nu\bigr],
\end{aligned}
\]
where $Q(t)$ counts arrivals through $t$, $\tau>0$ is fixed,
$0\leq a_0\leq A_0$, $|d_0|\leq D_0\log N$, and
$|G'|<1/(2s)$ wherever the derivative exists.
Fix a deterministic base height $Y$ and a terminal interval
$[y,y+w]$ of fixed positive bounded width.  The coarse height is
$Y+\sum_{j\leq Q_i}V_j$ at $t_i$.
If both height gaps $B(0)-Y$ and $B(T)-y$ are between fixed
positive multiples of $\log N$, the coarse probability of
staying below $B$ and finishing in that interval is at least
$N^{-\alpha_{\kappa,s}-\varepsilon}$ on typical count vectors.
This is simultaneous over the indicated gaps and coefficients.
\item
Both conclusions remain valid with a globally fixed collection
of block noises whose cumulative absolute value is at most
$\epsilon_N\to0$, and with a terminal weight bounded above and
below by fixed positive constants on the endpoint event.
The noises and the counted path functional cannot be chosen
depending on an angle pair or its separation time.
\end{enumerate}
These assertions concern typical random count vectors, not
arbitrary prescribed count allocations.
\end{lemma}

\begin{proof}
We specify the order of conditioning.
Before fixing the counts, the ordered clock points
$T_1<\cdots<T_N$ have $N+1$ exponential spacings conditioned on
their total $T$.  Theorems~\ref{rr:excursion} and
\ref{rr:shifts} apply at this stage.  We then use deterministic
path containments to transfer their environmental typicality
to counts.

For an upper comparison let $Q_i=Q(t_i)$ and put
\[
 S_i^{\rm c}=s\sum_{j\leq Q_i}V_j-\lambda Q_i-t_i .
\]
If $Q_{i-1}<j\leq Q_i$, then
\begin{equation}\label{rr:upper-containment}
 S_j\leq S_{i-1}^{\rm c}
       +(Q_i-Q_{i-1})\max\{s\log2-\lambda,0\}.
\end{equation}
The elapsed clock has a favorable sign.
Thus a coarse event is contained in a fine event with an
enlarged $O(\log N)$ buffer.  At the last arrival,
$S_N=S_m^{\rm c}+E_{N+1}$.
For a suffix starting at $t_i$, the corresponding mark-index
start is shifted by $t_i-T_{Q_i}$; a late prefix has the
analogous endpoint shift.  On the common event
$\max_jE_j\leq C_1\log N$, supplied by Lemma~\ref{rr:gaps},
all such shifts are controlled simultaneously, including
empty blocks.  Enlarging fixed polynomial cutoffs therefore
permits Lemma~\ref{rr:shifts}.  If the last block contains a
mark, its unused final gap is already bounded by its width,
but no such positivity assumption is required.

For the lower comparison, a height constant between arrivals
lies below every later value of $B$ until the next arrival,
because $B$ increases there.  Hence constraints immediately
after every arrival imply every coarse endpoint constraint,
even across empty blocks.  We verify that the physical-clock
fine event has the exponent in Theorem~\ref{rr:excursion}.
In $S$ coordinates its barrier is
\[
 h-sd_0t/T-sa_0
 \bigl[(\min(t,T-t)+\tau)^\nu-\tau^\nu\bigr],
 \qquad h=s(B(0)-Y).
\]
Put $d=sd_0$, $h'=s(B(T)-y)$, and $z=h-d-h'$.
Its endpoint interval in $S$ coordinates at the last arrival is
$[z+E_{N+1},z+E_{N+1}+sw]$.

Uniformly in the conditioned model,
\[
 \max_j|T_j-jT/N|=O_{\mathbb P}(\sqrt N).
\]
The linear barrier error is therefore $o_{\mathbb P}(1)$.
Also, for a fixed sufficiently large $C_2$ and any $u_N\to\infty$,
\begin{equation}\label{rr:clock-curve}
 \min(T_j,T-T_j)\leq C_2\min(j,N-j)+u_N
\end{equation}
with probability tending uniformly to one.
Indeed, \eqref{rr:gap-envelope}, applied with $n_0=N+1$ and
$u=u_N/2$, bounds $T_j$ by $C_0j+u_N/2$ and $T-T_j$ by
$C_0(N+1-j)+u_N/2$ simultaneously.  For large $N$ the
additional constant $C_0$ is at most $u_N/2$, which proves
\eqref{rr:clock-curve} with $C_2=C_0$.

Take $u_N=\log\log N$.
Subadditivity of $t^\nu$ bounds the physical curved loss by a
fixed multiple of $\min(j,N-j)^\nu$ plus $o(\log N)$.
The linear and curved errors can thus be paid by reducing the
index-clock initial buffer by a deterministic $o(\log N)$,
outside an environmental exception of probability tending to
zero.  For the lower endpoint retain the additional event
$E_{N+1}\leq c_*\log N$, with $c_*>0$ chosen small relative to
the minimum end buffer.  Its failure tends uniformly to zero
by Lemma~\ref{rr:gaps}.  More explicitly, let $R_N=o(\log N)$ be a deterministic bound
for the linear and extra curved errors outside the chosen exception.
The stricter index barrier has initial buffer $\bar h=h-R_N$,
linear coefficient $d$, and a fixed enlarged curve coefficient.
Its terminal buffer is
$\bar h'=h'-R_N-E_{N+1}$, since
$\bar h-d-\bar h'=z+E_{N+1}$.
Taking $c_*$ small makes both $\bar h$ and $\bar h'$ positive
multiples of $\log N$.  The uniformity over endpoint locations
and buffers in Theorem~\ref{rr:excursion} covers these
environment-dependent choices.  This proves the fine lower
estimate.

For clarity, no estimate is averaged against a rare mark event.
Let $p_{\rm c}(\boldsymbol q)$ be the coarse mark probability.
For every refinement on which the preceding containment holds,
an upper violation by $p_{\rm c}$ implies a violation of a
simultaneous fine upper bound.  A lower violation similarly
implies a violation of the fine lower bound.
Consequently
\begin{equation}\label{rr:count-transfer}
\begin{aligned}
 \mathbb P_{\rm counts}(\hbox{bad count vector})
 \leq{}&
 \mathbb P_{\rm gaps}(\hbox{bad fine reference bound})\\
 &+\mathbb P_{\rm gaps}(\hbox{containment exception})\\
 &+\mathbb P_{\rm counts}(\hbox{irregular count vector}),
\end{aligned}
\end{equation}
where the last term is used only for the stated count/width
bounds.  Fixed multiplicative constants are absorbed by
first using $\varepsilon/2$ in the fine estimates.
This proves the claimed count typicality and explains why
the gamma-density lemma was used before conditioning on
individual counts.

For lower bounds with noises, shrink the barrier by
$\epsilon_N$ and take a fixed inner terminal interval.
Every admissible noise realization then satisfies the original
event.  For upper bounds enlarge the barrier and terminal
interval instead.  The positive logarithmic buffers absorb
these shifts, and fixed bounded terminal weights do not alter
the exponent.  This proves the final assertion.
\end{proof}

For harmonic lengths above $j_0\asymp e^r$, the preceding lemma
uses the cumulative harmonic clock
$U(j)=\sum_{j_0<l\leq j}l^{-1}$.
Randomize uniformly inside the clock cell belonging to each
sampled integer.  The points are then exactly uniform clock
points, before the block counts are fixed, and
\[
 |U_{\rm point}-(\log j-r)|\leq C e^{-r}
\]
after the harmless initial endpoint adjustment.
Thus harmonic lengths are not asserted to be literally uniform
logarithmic positions.  Their clock discrepancy changes the
front and the Lipschitz barriers by $O(e^{-r})$ and is absorbed
in Lemma~\ref{rr:counts}.  Terminal widths in height units are
multiplied by $s$ in \eqref{rr:coordinates}; they remain in a
fixed positive compact interval.  This completes the reference
input without any spectral localization or branching theorem.

%% file: harmonic_core.tex
\section{Harmonic kernels and coarse localization}
\label{rh:section}

Write $e(t)=\exp(2\pi it)$ and
$A(s)=\Gamma(1+s)/\Gamma(1+s/2)^2=\exp\lambda(s)$.
All tilts below are real and positive. Put
$\phi_s(t)=|1-e(t)|^s$, assigning value zero at integers. A harmonic
sample in $I(a,c)=\{j:e^a<j\le e^c\}$ has probabilities $1/(jH)$,
where $H=\sum_{j\in I(a,c)}j^{-1}$. Define
\[
 B_{a,c}(\boldsymbol s;\boldsymbol t)
 =H^{-1}\sum_{j\in I(a,c)}j^{-1}\prod_{v=1}^d\phi_{s_v}(jt_v),
 \qquad d\in\{1,2\}.
\]
This section proves its analytic inputs directly. The only probabilistic
input for coarse localization is the independently proved reference
excursion estimate in Lemmas~\ref{rr:excursion} and~\ref{rr:counts}.
In particular, no spectral maximum theorem is used to prove the
bootstrap proposition below.

\begin{lemma}[Fourier and harmonic-sum bounds]\label{rh:fourier}
For $0<p_-<p_+<\infty$, uniformly for $s\in[p_-,p_+]$,
\[
 |\widehat\phi_s(j)|\le C(1+|j|)^{-1-p_-},\qquad
 \sum_j|\widehat\phi_s(j)|\le C,
 \qquad
 \sum_{|j|>R}|\widehat\phi_s(j)|\le CR^{-p_-}.
\]
For real $s>0$ the coefficients are
\begin{equation}\label{rh:coefficients}
 \widehat\phi_s(j)=
 \frac{(-1)^j\Gamma(1+s)}
 {\Gamma(1+s/2-j)\Gamma(1+s/2+j)}.
\end{equation}
Moreover, for $u\notin\mathbb Z$,
\begin{equation}\label{rh:abel}
 \left|\sum_{j\in I(a,c)}\frac{e(ju)}j\right|
 \le \frac{Ce^{-a}}{\|u\|_{\mathbb R/\mathbb Z}}.
\end{equation}
\end{lemma}
\begin{proof}
The beta integral gives $\widehat\phi_s(0)=A(s)$.
The derivative identity
\[
 \bigl((1-e(t))\phi_s(t)\bigr)'
 =-\pi i\{s+(s+2)e(t)\}\phi_s(t)
\]
has integrable terms and vanishing boundary values. Taking its Fourier
coefficient gives
$(s/2+j)\widehat\phi_s(j)=-(1+s/2-j)\widehat\phi_s(j-1)$.
This recurrence, symmetry in $j$, and the value at zero prove
\eqref{rh:coefficients}, with reciprocal gamma zeros interpreted
literally. Fix an integer $j_*>1+p_+/2$. For $j>j_*$ the recurrence
gives
\[
 |\widehat\phi_s(j)|
 =|\widehat\phi_s(j-1)|
   \left(1-\frac{s+1}{j+s/2}\right).
\]
The factors lie in $[0,1]$. Using $1-x\le e^{-x}$ and summing
$(j+s/2)^{-1}$ shows, uniformly on the compact tilt interval,
$|\widehat\phi_s(j)|\le Cj^{-1-s}$.
The finitely many earlier coefficients are bounded by $A(s)$;
symmetry and summation prove all three bounds. This argument also
includes even integer tilts, for which the coefficient sequence
eventually vanishes. Partial sums of $e(ju)$ are bounded by $C/\|u\|$;
summation by parts proves \eqref{rh:abel}.
\end{proof}

We next separate two tasks: estimating the mass of small phase cells,
and turning those estimates into comparisons of tilted paths. The
second task is the same for separated angles and for a true diagonal.

\begin{lemma}[Phase cells for separated angles]\label{rh:cells}
Let $d\in\{1,2\}$, $0<\sigma<\rho/4<1/8$, and $R\ge\sigma^{-1}$.
For a product $C$ of $d$ circle intervals of length $\rho$, there
are trigonometric polynomials $P_C^-\le\ind_C\le P_C^+$, with
frequencies $|h_v|\le R$, such that
\begin{equation}\label{rh:cell-polynomials}
 \int P_C^\pm=\rho^d+O\bigl(\sigma\rho^{d-1}+(\sigma R)^{-1}\bigr),
 \qquad
 \sum_{\boldsymbol h}|\widehat P_C^\pm(\boldsymbol h)|
       \le C_d\{1+\log(1/\sigma)\}^d.
\end{equation}
Suppose that the observation angles satisfy
\begin{equation}\label{rh:separation}
 \|\boldsymbol h\cdot\boldsymbol t\|\ge e^{-a+\Delta}
 \quad(0\ne\boldsymbol h\in\mathbb Z^d,\ |h_v|\le R).
\end{equation}
If $\pi$ is the phase law of one harmonic sample from $I(a,c)$,
then
\begin{equation}\label{rh:cell-mass}
 \left|\frac{\pi(C)}{\rho^d}-1\right|
 \le C_d\left\{\frac\sigma\rho+
   \frac1{\sigma R\rho^d}+
   \frac{e^{-\Delta}\{1+\log(1/\sigma)\}^d}{H\rho^d}\right\}.
\end{equation}
\end{lemma}
\begin{proof}
For a circle interval $I$ of length $\rho$, convolve the indicators
of its $\sigma$-contraction and $\sigma$-enlargement with the uniform
density on $[-\sigma,\sigma]$. The resulting continuous functions
$f_I^-$ and $f_I^+$ satisfy
$0\le f_I^-\le\ind_I\le f_I^+\le1$, and have integrals
$\rho-2\sigma$ and $\rho+2\sigma$. The inequalities hold at the
endpoints as well, regardless of the convention for the interval.
Their nonzero Fourier coefficients have absolute value at most
\[
 C\min(\rho+\sigma,|h|^{-1})
       \min(1,(\sigma|h|)^{-1}).
\]
Consequently their Fourier coefficient sums are at most
$C\{1+\log(1/\sigma)\}$ and their tails beyond $R$ at most
$C/(\sigma R)$. Products of the nonnegative continuous functions
bracket $\ind_C$. Replace each factor by its Fourier sum through
$R$. Each replacement changes the product uniformly by at most
$C_d/(\sigma R)$, because the original factors are bounded by one
and $\sigma R\ge1$. Add or subtract this error as a constant term.
This constructs the claimed polynomial bounds and proves
\eqref{rh:cell-polynomials}. In particular, no product of signed
polynomial minorants is used. Finally, \eqref{rh:abel} bounds every
nonzero Fourier coefficient of $\pi$ in this box by
$CH^{-1}e^{-\Delta}$. Integrate the two polynomial bounds against
$\pi$ and divide by $\rho^d$ to obtain \eqref{rh:cell-mass}.
\end{proof}

\begin{lemma}[Positive cells give noisy path comparisons]
\label{rh:cell-transfer}
Fix $d\in\{1,2\}$ and a compact positive tilt interval
$[p_-,p_+]$. Partition the phase torus into cubes of side
$\rho=1/m$, and suppose that probability measures $\pi_i$ satisfy
\[
 \left|\frac{\pi_i(C)}{\rho^d}-1\right|\le\eta_i\le\tfrac12
 \quad\hbox{for every cube }C.
\]
Let $\Psi$ be an integer matrix with determinant $\pm1$ and entries
bounded by one. A sample $x$ has heights
$f((\Psi x)_v)$, where $f(x)=\log|1-e(x)|$; the values at roots
may be assigned arbitrarily after tilting. In block $i$ take $q_i$
independent samples, tilt by
$w(x)=\prod_v\phi_{s_v}((\Psi x)_v)$ with $s_v\in[p_-,p_+]$,
and add an independent uniform noise on $[-\delta,\delta]$ to each
block-height coordinate. All samples and noises are independent
across blocks. Let $\mu$ be the resulting joint law,
and $\nu$ its counterpart with Haar phase samples. Put
\[
 T=\sum_iq_i,\qquad M_i=\int w\,d\pi_i,\qquad
 M_*=\int w\,dx=\prod_v A(s_v).
\]
For $0<\rho\le\tau/10$ and sufficiently small $\tau$, one can take
\begin{equation}\label{rh:transfer-errors}
 E=C\left\{\sum_iq_i\eta_i+
             T\bigl(\rho/\tau+\tau^{p_-+1}\bigr)\right\},
 \qquad
 \varepsilon=C(1+e^E)T
     \left\{\tau^{p_-+1}+\frac\rho{\tau\delta}\right\},
\end{equation}
so that
\begin{equation}\label{rh:transfer-comparison}
 \left|\log\frac{\prod_iM_i^{q_i}}{M_*^T}\right|\le E,
 \qquad
 e^{-E}\nu(F)-\varepsilon\le\mu(F)
       \le e^E\nu(F)+\varepsilon
\end{equation}
for every measurable $0\le F\le1$. The same assertion holds after
any fixed measurable transformation of the noisy block vector,
including passage to partial-sum paths.
\end{lemma}
\begin{proof}
Call a cube bad if it meets a strip
$\| (\Psi x)_v\|\le2\tau$. The union of these cubes has Haar
measure $O(\tau)$, since $\Psi$ preserves Haar measure and
$\rho\le\tau/10$. On each bad cube at least one coordinate is
within $C\tau$ of a root throughout the cube. Therefore $w$ is
at most $C\tau^{p_-}$ there. The cell-mass assumption gives
\[
 \int_{\rm bad} w\,dx+\int_{\rm bad}w\,d\pi_i
          \le C\tau^{p_-+1}.
\]
On every good cube, all relevant coordinates stay away from roots,
and $|f'(x)|\le C/\tau$. Thus each height oscillates by at most
$C\rho/\tau$, and the ratio of any two positive weights on the
cube lies between $e^{-C\rho/\tau}$ and $e^{C\rho/\tau}$.
For every good cube it follows that
\[
 e^{-C(\eta_i+\rho/\tau)}\int_Cw\,dx
 \le\int_Cw\,d\pi_i
 \le e^{C(\eta_i+\rho/\tau)}\int_Cw\,dx.
\]
Because $M_*$ has a uniform positive lower bound, summation over
good cubes and the preceding bad-cube estimate give
\[
 |\log(M_i/M_*)|
 \le C\{\eta_i+\rho/\tau+\tau^{p_-+1}\}.
\]
In particular the $M_i$ are uniformly bounded below. After division
by the normalizers, the same exponential bounds, with a larger
constant, compare the tilted probabilities of every good cube.
Their products compare the probabilities of every all-good sample
label sequence by $e^{\pm E}$. They also prove the first assertion
of \eqref{rh:transfer-comparison}. Under either tilted law the
probability of any bad sample is at most $CT\tau^{p_-+1}$.

Conditional on the same good labels, couple the actual and Haar
phase samples within their assigned cubes. This changes the height
in coordinate $v$ of block $i$ by at most $Cq_i\rho/\tau$.
For the uniform density $h_\delta$ on $[-\delta,\delta]$,
\[
 \|h_\delta(\cdot-z)-h_\delta\|_1
       \le \min(2,|z|/\delta).
\]
Changing the independent noise factors one at a time therefore
bounds the conditional total variation error by
$CT\rho/(\tau\delta)$. Integrating these conditional bounds over
the good label sequences, using their positive multiplicative
comparison, and then adding back the bad sequences proves the
second assertion of \eqref{rh:transfer-comparison}. The factor
$1+e^E$ covers either direction of this comparison. Composition
with a measurable path map does not increase total variation.
\end{proof}

\begin{lemma}[Smoothed comparison for small or coarse blocks]
\label{rh:smooth}
Fix $h_0,C_0>0$ and a compact set of positive real tilts. For every
$J>0$ there are $u_2,u_3>0$ with the following property. Let
$L\to\infty$, $0\le q\le C_0L$, $H\ge h_0$,
\[
 \delta=L^{-10},\quad
 R=\lceil L^{u_2}\rceil,\quad\Delta=u_3\log L.
\]
Assume \eqref{rh:separation}. Take the $d$ heights of $q$ harmonic
samples, tilt their joint law by $\exp(\sum_v s_vX_v)$, and normalize.
Convolve each coordinate with an independent uniform on
$[-\delta,\delta]$. This law differs in total variation by
at most $L^{-J}$ from the product of the corresponding iid tilted
log-sine sums with the same noises. In addition,
\[
 B_{a,c}(\boldsymbol s;\boldsymbol t)^q
 =\prod_v A(s_v)^q\{1+O(L^{-J})\}.
\]
The assertion includes $q=0$. Products of at most $CL$ such kernels
have comparison error at most $CL^{1-J}$ for every test in $[0,1]$.
\end{lemma}
\begin{proof}
Let $p_-$ be the smallest tilt in the compact set. Choose constants
\[
 B(p_-+1)>J+8,\quad Q>B+J+25,\quad
 u_2>3Q+2J+30,\quad u_3>2Q+J+25,
\]
and put $\tau=L^{-B}$, $\rho=1/\lceil L^Q\rceil$ and
$\sigma=\rho L^{-J-12}$. By \eqref{rh:cell-mass}, the relative
error in every phase cell is $O(L^{-J-7})$; the choices have
slack to absorb the logarithmic factors and fixed constants.
Apply Lemma~\ref{rh:cell-transfer} with $\Psi$ the identity and one
block of $q$ samples. Both $E$ and $\varepsilon$ in
\eqref{rh:transfer-errors} are $O(L^{-J-6})$. Consequently
\eqref{rh:transfer-comparison}, applied to indicators, gives total
variation at most $L^{-J}$ for all sufficiently large $L$, as well
as the real normalizer estimate. For $q=0$ both laws are exactly
the prescribed noises. Finally change one probability-kernel
factor at a time in a product; each change costs its total
variation, independently of the bounded test.
\end{proof}

\begin{lemma}[Entropy envelope]\label{rh:entropy}
Put
\[
 b_s=\begin{cases}0,&0<s\le2,\\
 (s-1)\log2-\lambda(s),&s>2.
 \end{cases}
\]
For every harmonic block and angle,
\begin{equation}\label{rh:envelope}
 B_{a,c}(s;t)\le A(s)\exp(b_s+C/H).
\end{equation}
At the critical tilt $s=s_\kappa$, $\kappa b_s<1$ with a common
strict gap on each positive compact $\kappa$ interval. Also $b_s$
is nondecreasing, so the same gap applies to every $0<\sigma\le s$.
For $0<s\le2$ the constant in \eqref{rh:envelope} is uniform down
to $s=0$.
\end{lemma}
\begin{proof}
For any integer interval $I$,
$\sum_{j\in I}\cos(2\pi ju)/j\ge-C$, uniformly in $u$.
Indeed the terms through $(6\|u\|)^{-1}$ are nonnegative, and
\eqref{rh:abel} bounds the rest by a constant; if this cutoff is less
than one, start the same estimate at one. For $0<s\le2$, recurrence
\eqref{rh:coefficients} makes every nonzero coefficient nonpositive.
Their absolute sum equals $A(s)$, since the absolutely convergent
Fourier series has value zero at the origin. The cosine lower bound
thus gives $B(s;t)\le A(s)(1+C/H)$. For $s>2$ use
$\phi_s\le2^{s-2}\phi_2$, obtaining
$B(s;t)\le2^{s-1}+C/H$, which is \eqref{rh:envelope}.

Differentiating the gamma duplication identity and using the integral
for the digamma difference gives
\[
 \log2-\lambda'(s)=\int_0^\infty\frac{e^{-sx}}{1+e^x}\,dx,
 \qquad 0<s\{\log2-\lambda'(s)\}<\tfrac12<\log2.
\]
Thus for $s>2$ the critical equation gives
$\kappa b_s=1+\kappa\{s(\log2-\lambda'(s))-\log2\}<1$.
For $s\le2$ the assertion is immediate. Continuity makes the margin
uniform on compact intervals. Finally $b'_s=\log2-\lambda'(s)>0$
above two, and the two definitions agree at two.
\end{proof}

\begin{lemma}[Integrated moment with an explicit initial cutoff]
\label{rh:integrated}
Let $r=r_*\log L\le a\le C L$, and partition $(r,a]$ into equal
logarithmic widths satisfying
\[
 A_0\log L/2\le\omega\le A_0\log L.
\]
Condition the independent harmonic samples on counts $q_i$ satisfying
$q_i\le(\kappa+\eta)\omega$ and $\sum q_i\le CL$. Suppose that
$\sigma$ lies in a fixed positive compact interval and
$(\kappa+\eta)b_\sigma\le1-\varepsilon_0$ for a fixed
$\varepsilon_0>0$. For sufficiently large fixed $r_*$,
\begin{equation}\label{rh:integrated-bound}
 \mathbb E\int_{\mathbb R/\mathbb Z}
       e^{\sigma X_{\rm mid}(t)}\,dt
 \le A(\sigma)^{\sum q_i}(1+o(1)).
\end{equation}
The bound is uniform over these counts and every truncated prefix.
At a smaller horizon $a$, adjoining a root block of $O(\log a)$
samples below $\exp(C\log a)$ consequently gives the bound
$A(\sigma)^q\exp(O(\log a))$ for the entire field.
\end{lemma}
\begin{proof}
Use Lemma~\ref{rh:smooth} only for its real normalizers and put
\[d_R(t)=\min_{1\le|h|\le R}\|ht\|.\]
On
$d_R(t)\ge e^{-r+\Delta}$ all blocks are separated and their product
normalizer is $A(\sigma)^Q(1+o(1))$. For the other angles use classes
\[
 e^{-a_j+\Delta}\le d_R(t)<e^{-a_{j-1}+\Delta},
\]
placing all smaller distances in the final class. Each class has
measure at most $CR e^{-a_{j-1}+\Delta}$. The blocks after $j$ are
separated. By \eqref{rh:envelope} the normalized prefix product is
at most $\exp\{\rho(a_j-r)\}$ for a fixed $\rho<1$: the errors
$Cq_i/H_i=O(1)$ per block are absorbed by the strict entropy gap
because $\omega\to\infty$. Thus the total normalized bad-angle
integral is at most
\[
 CL^{u_2+u_3}e^{-r+\rho\omega}
       \sum_{j\ge1}e^{-(1-\rho)(j-1)\omega}=o(1)
\]
when $r_*>u_2+u_3+A_0+2$. The same geometric bound applies to
every truncated prefix, including its last class with empty suffix.
For the root block use $X_{\rm root}\le q_{\rm root}\log2$;
its normalization loss is $\exp(O(\log a))$.
\end{proof}

\begin{lemma}[Polynomial maximum and moment]\label{rh:bernstein}
For a nonzero polynomial $P$ of degree $d\ge1$,
\[
 \sup_{|z|\le1}|P'(z)|\le ed\|P\|_\infty,
 \qquad
 \int_{\mathbb R/\mathbb Z}|P(e(t))|^s\,dt
       \ge\frac{c_s}{d}\|P\|_\infty^s.
\]
There is an arc of length $c/d$ around a boundary maximizer on which
$|P|\ge\|P\|_\infty/2$. Constants in the moment inequality are
uniform on compact positive $s$ intervals.
\end{lemma}
\begin{proof}
The maximum principle applied outside the unit disk to $P(z)/z^d$
gives $|P(z)|\le|z|^d\|P\|_\infty$ for $|z|\ge1$.
Cauchy's derivative estimate on a circle of radius $1/d$ about any
point in the closed disk therefore gives the first bound. Integration
along the unit circle supplies the arc assertion and its moment bound.
\end{proof}

\begin{proposition}[Coarse harmonic localization]\label{rh:localization}
Let $B\to\infty$, $a=\log B$, $H_B=\sum_{j\le B}j^{-1}$,
and let $q$ be an integer with $\kappa=q/H_B$ in a fixed positive
compact interval. For independent harmonic lengths
\[
 \mathbb P(J_v=j)=\frac1{jH_B},\qquad
 X(t)=\sum_{v=1}^q\log|1-e(J_vt)|,
\]
put $s=s_\kappa$. There is a fixed $C$ such that, uniformly in $q$,
\begin{equation}\label{rh:localization-result}
 \mathbb P\left(\left|\max_t X(t)
             -\frac{a+q\lambda(s)}s\right|>C\log a\right)=o(1).
\end{equation}
This has the following conditional version. Partition at an initial
cutoff $r=r_*\log a$ and then into logarithmic blocks of widths
comparable to $A_0\log a$ as in the proof. There are sets of these
coarse count vectors of probability $1-o(1)$ on which the conditional
failure probability in \eqref{rh:localization-result} tends uniformly
to zero. No conditioning on an arbitrary finer allocation is asserted.
\end{proposition}
\begin{proof}
Write $\ell_a=\log a$ and use the smoothing scale $L=a$.
Choose $\eta>0$ so that $(\kappa+\eta)b_s<1$ uniformly.
With $A_0$ large enough, Chernoff bounds and a union over at most $Ca$
blocks give, with probability $1-o(1)$,
\begin{equation}\label{rh:coarse-counts}
 \begin{gathered}
 \tfrac12\kappa\omega\le q_i\le(\kappa+\eta)\omega,
 \qquad q_0=O(r),\\
 |q_0-\kappa r|\le C\sqrt{r\log r},\qquad
 Q-\kappa(a-r)=O(\ell_a).
 \end{gathered}
\end{equation}
Here $q_0$ is the root count, $S_i=\sum_{j\le i}q_j$ and $Q=S_m$.
The last assertion follows also from $q=\kappa H_B$ and
$H_B=a+O(1)$. Harmonic widths differ from their logarithmic widths
by a total $O(e^{-r})$.

Choose constants $G=g_*\ell_a$ and $r=r_*\ell_a$, and define
\[
 F_i=\frac{a_i+\lambda(s)S_i}s,
 \quad B_i=F_i-G,
 \quad y=\frac{a-r+\lambda(s)Q}s.
\]
For a cumulative path $w=(w_i)_{i\le m}$ write
\[
 \mathcal A=\{w:w_i\le B_i\ (1\le i\le m),\ w_m\in[y,y+1]\}.
\]
For independent $s$-tilted reference marks, with the fixed bounded
block noises, let $W_i$ be the cumulative height and put
\[
 p=\mathbb E_s\left[e^{-s(W_m-y)}
     \mathbf1_{\mathcal A}(W)\right].
\]
The reference excursion and count-transfer lemmas give, on a further
set of coarse counts of probability $1-o(1)$,
\begin{equation}\label{rh:reference-cost}
 p\ge a^{-C_*},\qquad C_* =\sup_{\kappa\in K}\alpha(\kappa)+1.
\end{equation}
Here $K$ may be slightly enlarged once. We spell out why this use is
independent of $g_*,r_*$. After auxiliary harmonic-clock refinement,
the centered walk is $s\sum_{v\le j}V_v-\lambda(s)j-T_j$ and its
flat barrier is $r-sG$. The exact clock total differs from
$(Q+1)/\kappa$ by $O(\log a)$ by \eqref{rh:coarse-counts}.
Every last coarse block has a positive count, so its unused terminal
spacing is at most $\omega+o(1)$. Choose
\begin{equation}\label{rh:buffer-choice}
 r_*-s_{\max}g_*>2A_0+2.
\end{equation}
Then both the initial and final gaps are positive fixed multiples
of $\log a$, including every such unused spacing. Use the inner
terminal interval $[y+1/4,y+3/4]$ and a fixed stronger barrier to
absorb the $O(a^{-9})$ total noise. Between arrivals the height barrier
increases. Fine constraints therefore imply every coarse constraint.
The simultaneous reference lower bound and deterministic count transfer
prove \eqref{rh:reference-cost}. Its exponent may be fixed as displayed
before choosing the positive buffer coefficients; only the convergence
threshold depends on those fixed coefficients.

Choose $J>2C_*+12$ and Fourier exponents $u_2,u_3$ from
Lemma~\ref{rh:smooth}. Choose $g_*$ and then $r_*$ sufficiently large
that \eqref{rh:buffer-choice} holds and
\begin{equation}\label{rh:bootstrap-order}
 s_{\min}g_*>1+2u_2+u_3+A_0+2C_*+4,
 \qquad r_*>2u_2+u_3+A_0+4.
\end{equation}
These choices are noncircular because $C_*$ is independent of
$g_*,r_*$. No multiscale count-energy estimate is required.

Let $\mathcal D_a=\{t:\|ht\|\ge e^{-r+\Delta},\ 1\le|h|\le R\}$.
Its complement has measure $o(1)$. We prove a middle-field lower bound
uniformly over every measurable $D\subset\mathcal D_a$ with $|D|\ge1/2$.
For the actual middle field, define
\[
 \begin{gathered}
 \widetilde W_i(t)=\sum_{j\le i}(X_j(t)+\zeta_j),\qquad
 Z_D=\int_D Z(t)\,dt,\\
 Z(t)=\mathbb E_\zeta\left[e^{-s\sum_i\zeta_i}
                   \mathbf1_{\mathcal A}(\widetilde W(t))\right].
 \end{gathered}
\]
All noises are fixed in law before an angle pair is considered.
Under the one-point tilt, the inverse noise factor combines with the
inverse mark tilt into the bounded terminal test defining $p$.
Consequently
\begin{equation}\label{rh:first-moment}
 \mathbb E Z(t)=(1+o(1))\mu p,
 \qquad \mu=e^{-sy+\lambda(s)Q}=e^{-a+r},
\end{equation}
uniformly on $\mathcal D_a$. Indeed the normalized product error is
$O(a^{1-J})=o(p^2)$.

Put $d_R(t,u)=\min_{(h,k)\ne(0,0),\ |h|,|k|\le R}\|ht+ku\|$.
On $d_R(t,u)\ge e^{-r+\Delta}$ the two-point comparison gives
$\mathbb E Z(t)Z(u)=(1+o(1))\mu^2p^2$. For the remaining pairs,
classify by
$e^{-a_j+\Delta}\le d_R(t,u)<e^{-a_{j-1}+\Delta}$,
including all smaller distances in the last class. Its Haar measure
is at most $CR^2e^{-a_{j-1}+\Delta}$, since every nonzero integer
linear map from the two-dimensional torus preserves Haar measure.
On both path events, the noisy terminals are at least $y$ and the
$u$-prefix is at most $B_j$. Exponential Markov and that prefix bound
give the pointwise estimate
\[
 \mathbb E Z(t)Z(u)
 \le\exp\{-2sy+sB_j+\lambda(s)S_j
                    +2\lambda(s)(Q-S_j)+o(1)\}.
\]
To verify the factorization, the resulting exponential contains only
the full tilt of the $t$-prefix and of both suffixes. These involve
disjoint blocks. The prefix is one-point separated, the suffix is
two-point separated, and in the last class the suffix is empty.
Noise errors are $O(ma^{-10})$. Since $sB_j=a_j+\lambda(s)S_j-sG$,
the normalized near contribution is bounded by
\[
 C\sum_{j\le m}R^2e^{-a_{j-1}+\Delta}e^{a_j-sG}p^{-2}
 \le C a^{1+2u_2+u_3+A_0-s_{\min}g_*+2C_*}=o(1).
\]
The constants in this calculation are independent of $D$. It follows
from \eqref{rh:first-moment} and the far-pair bound that
$\mathbb E Z_D^2/(\mathbb E Z_D)^2=1+o(1)$.
Chebyshev's inequality therefore gives, uniformly over $D$,
\begin{equation}\label{rh:middle-low}
 \mathbb P\{\sup_{t\in D}X_{\rm mid}(t)<y-o(1)\}=o(1).
\end{equation}

For the upper bound, Lemmas~\ref{rh:integrated} and~\ref{rh:bernstein},
with middle degree at most $QB\le CaB$, imply
\[
 \mathbb P\left\{\max_tX_{\rm mid}(t)>
          \frac{a+\lambda(s)Q}s+C_1\log a\right\}
 \le Ca^{1-s_{\min}C_1}=o(1)
\]
for fixed $C_1$ large enough.
It remains to restore the root field $X_0$, without a pointwise lower
bound. For every deterministic root configuration,
\[
 \int(-X_0(t))_+\,dt\le c_Vq_0,\qquad
 c_V=\int(-\log|1-e(t)|)_+\,dt<\infty.
\]
Choose $C_2$ so that $\{X_0\ge-C_2r\}\cap\mathcal D_a$ has measure
at least one half whenever $q_0\le Cr$. Condition on the root samples
and apply the uniform estimate \eqref{rh:middle-low} to this random
set; the middle samples remain independent. The opposite deterministic
bound is $X_0\le q_0\log2$. Since $r,q_0=O(\log a)$, these two
estimates and $q=q_0+Q$ prove \eqref{rh:localization-result}.
Every conditional error above is uniform on the declared good count
set, proving the conditional assertion as well.
\end{proof}

\begin{corollary}[Changing the low-horizon tilt]\label{rh:outer-tilt}
Suppose the preceding $q,a$ also satisfy
$|q-\kappa_0a|\le C\sqrt{a\log a}$, with $\kappa_0$ in a fixed
positive compact interval. Then the center in
\eqref{rh:localization-result} may be replaced by
$(a+q\lambda(s_{\kappa_0}))/s_{\kappa_0}$, at a cost $O(\log a)$.
\end{corollary}
\begin{proof}
The function $f(t)=(a+q\lambda(t))/t$ has its minimum at
$s_{q/a}$. On the relevant compact tilt interval $f''=O(a)$ and
$\kappa\mapsto s_\kappa$ is continuously differentiable with bounded
derivative, by differentiating
$\kappa(s\lambda'(s)-\lambda(s))=1$.
Thus $|s_{\kappa_0}-s_{q/a}|=O(\sqrt{\log a/a})$ and the excess
over the minimum is $O(\log a)$. Since $H_B=a+O(1)$, replacing
$s_{q/a}$ by $s_{q/H_B}$ changes the minimum comparison by $O(a^{-1})$.
\end{proof}

%% file: arithmetic.tex
\section{Positive comparisons and arithmetic pairs}\label{ra:section}

We keep the sampling law fixed and compare its tilted kernels.
For a block $I=\{j:N<j\le e^\omega N\}$, with
$1/2\le\omega\le1$, a harmonic sample has probability $1/(jH)$,
where $H=\sum_{j\in I}1/j$. Write
$\phi_s(x)=|1-e(x)|^s$. Every noise used below is uniform on
$[-L^{-10},L^{-10}]$.
The grid of blocks, and one noise for each block and observation copy,
are fixed before an angle pair is selected.

\subsection{Positive phase cells near a true diagonal}

\begin{lemma}\label{ra:phase}
Fix a positive compact interval of tilts and an accuracy $J>0$.
There are fixed exponents $U,K>0$ with the following property.
Consider at most $CL$ consecutive blocks of widths in $[1/2,1]$,
with first frequency $N_0$ and counts $q_i$ satisfying
$\sum_iq_i\le CL$ and $M=\max_iq_i\le C\ell$.
Let $u=t+d$ or $u=-t+d$, where $D=|d|\in(0,1/2)$, and assume
\[
 N_0D\ge\ell^4,\qquad
 \min_{1\le |h|\le L^U}\|ht\|_\T\ge\epsilon_0,\qquad
 D/\epsilon_0\le L^{-K}.
\]
Under the joint positive tilt $(s,s)$, the vector of the two noisy
block paths has law $\mu$. Let $\nu$ be the law of two independent
tilted log-sine paths with the same counts and noises.
For every measurable $0\le F\le1$,
\begin{equation}\label{ra:positive}
 e^{-\epsilon_L}\nu(F)-CL^{-J}
 \le\mu(F)\le e^{\epsilon_L}\nu(F)+CL^{-J},
 \qquad
 \epsilon_L\le C M/(N_0D)+o(1)=o(1).
\end{equation}
The product normalizer divided by $A(s)^{2\sum q_i}$ lies between
$e^{-\epsilon_L}$ and $e^{\epsilon_L}$ after enlarging $\epsilon_L$.
All errors are uniform under the stated assumptions.
\end{lemma}

\begin{proof}
It suffices to treat $u=t+d$; the reflected case uses the same
phase cells. Let $s_{\min}$ be the smallest tilt in the compact
interval. Choose
\[
 B(s_{\min}+1)>J+8,\quad Q>B+J+25,\quad
 U>3Q+2J+30,\quad K>2Q+J+25,
\]
and put $\rho=1/\lceil L^Q\rceil$, $\tau=L^{-B}$,
$\sigma=\rho L^{-J-12}$ and $R=\lfloor L^U\rfloor$.

First we need a weighted rotation estimate on any integer subinterval
of $[N,e^\omega N]$. For an interval of phases of length $\rho$,
use the one-dimensional polynomial bounds
\eqref{rh:cell-polynomials}. Abel summation and
$\|ht\|\ge\epsilon_0$ give, on an integer interval $I'$ with first
index $j_0$,
\begin{equation}\label{ra:rotation}
 \sum_{j\in I'}\frac{\ind_{\{jt\in I_x\}}}{j}
 =\{\rho+O(\sigma+(\sigma R)^{-1})\}
      \sum_{j\in I'}\frac1j
   +O\!\left(\frac{\log(1/\sigma)}{j_0\epsilon_0}\right).
\end{equation}
The displayed equality denotes matching upper and lower error bounds.

Use phase coordinates $x=\{jt\}$ and $v=\{jD\}$.
For a continuous harmonic variable on $[N,e^\omega N]$, the
unnormalized density of $v$ is
\[
 g_N(v)=\sum_{ND<m+v\le e^\omega ND}\frac1{m+v}
       =\omega+O((ND)^{-1})
\]
uniformly in $v$, by integral comparison for $1/x$.
The preimage of a phase interval $I_v$ consists of intervals indexed
by $m$. The sum of the reciprocals of their left endpoints is
$O(D+1/N)$. Replacing each harmonic integral by its integer sum
therefore costs $O(D+1/N)$. Apply \eqref{ra:rotation} to each such
interval and use $H=\omega+O(1/N)$. Every phase cell has probability
$\rho^2(1+\theta)$, where
\begin{equation}\label{ra:cellerror}
 |\theta|\le C\left\{
 \frac1{ND}+\frac{\sigma}{\rho}
 +\frac1{\sigma R\rho}
 +\frac{D\log(1/\sigma)}{\rho^2\epsilon_0}\right\}.
\end{equation}
Here $1/N\le D$ follows from $ND\ge\ell^4$.
This estimate also permits rational $D$.

It remains to apply the common transfer lemma. The two observation
phases are obtained from $(x,v)$ by
\[
 \Psi(x,v)=(x,x+\operatorname{sgn}(d)v),
 \quad\hbox{or}\quad
 \Psi(x,v)=(x,-x+\operatorname{sgn}(d)v)
\]
in the reflected case. Both integer maps preserve Haar measure.
The bound \eqref{ra:cellerror} is at most $1/2$ for large $L$.
Its aggregate slow-phase error satisfies
\[
 \sum_i\frac{q_i}{N_iD}\le\frac{CM}{N_0D},
\]
since successive $N_i$ grow by a factor at least $e^{1/2}$. All
other cell errors, summed over the at most $CL$ samples, are
$O(L^{-J-6})$. Lemma~\ref{rh:cell-transfer}, with
$\delta=L^{-10}$, therefore gives
\[
 E\le\frac{CM}{N_0D}+O(L^{-J-6})=o(1),\qquad
 \varepsilon=O(L^{-J-6}).
\]
Here $M\le C\ell$ and $N_0D\ge\ell^4$ ensure $E=o(1)$, so its
exponential factor in \eqref{rh:transfer-errors} is bounded.
The lemma proves \eqref{ra:positive}, the normalizer comparison,
and the same statement for the two partial-sum paths.
\end{proof}

\subsection{A gap for other rational relations}

\begin{lemma}\label{ra:resonance}
Let $R$ be a fixed power of $L$ and assume
\[
 \|ht\|,\|hu\|\ge\epsilon_0\quad(1\le|h|\le2R^2),
 \qquad R\eta_0\le\epsilon_0/2.
\]
Suppose some nonzero $(p,q)$ with $|p|,|q|\le R$ satisfies
$\|pt+qu\|\le\eta_0$, but
$\|t-u\|,\|t+u\|>\eta_0$.
For a harmonic block starting at $N$ with mass $H\ge h_0>0$, once
$N\epsilon_0/R$ and the Fourier cutoff are sufficiently large,
uniformly on a positive compact tilt interval,
\begin{equation}\label{ra:gap}
 B_I(s/2,s/2;t,u)\le A(s)(1-d_*),\qquad d_*>0.
\end{equation}
For logarithmic horizons used below it suffices to take
$\log N\ge\log(1/\epsilon_0)+\log R+J_0\ell$ with fixed large $J_0$.
\end{lemma}
\begin{proof}
Both $p,q$ are nonzero. If $|h|,|l|\le R$ and $(h,l)$ is independent of $(p,q)$,
the determinant identity and the strong one-point condition give
\[
 \epsilon_0\le\|(pl-qh)t\|
       \le R\eta_0+R\|ht+lu\|,
 \qquad \|ht+lu\|\ge\epsilon_0/(2R).
\]
Let $f=\phi_{s/2}$. Its first Fourier coefficient is
\[
 c_1=-\frac{s}{s+4}A(s/2)\ne0.
\]
In the harmonic probability space of the block,
\[
 \langle f(jt),e(jt)\rangle=c_1+o(1),\qquad
 \|f(jt)\|_2^2=A(s)+o(1),
\]
and similarly for $u$. These uniform errors follow from
Lemma~\ref{rh:fourier} and Abel summation.

If the relation direction is not a diagonal, one of the two mode
families $(-1,l)$ and $(l,-1)$ contains no vector parallel to $(p,q)$.
Indeed parallel vectors in both families would have integer indices
with product one. The corresponding cross projection
$\langle f(ju),e(jt)\rangle$ or $\langle f(jt),e(ju)\rangle$
is therefore $o(1)$.

For a diagonal direction with a nontrivial cyclic alias, write the
relation as $d(t\pm u)$, $d\le R$. Since $\|t\pm u\|>\eta_0$
but $\|d(t\pm u)\|\le\eta_0$, the primitive phase is at distance
at least $(1-\eta_0)/d\ge1/(2R)$ from an integer. The exceptional
first cross mode is again small; other modes obey the determinant
bound. Projection on the selected modulus-one exponential shows
$\|f(jt)-f(ju)\|_2^2\ge |c_1|^2/4$.
Using
\[
 2\langle f(jt),f(ju)\rangle
 =\|f(jt)\|_2^2+\|f(ju)\|_2^2-\|f(jt)-f(ju)\|_2^2
\]
proves the assertion, with for example any common
$d_*\le\min_s |c_1(s)|^2/(64A(s))$.
\end{proof}

\subsection{Padding a low-frequency maximum arc}

\begin{lemma}\label{ra:padding}
Let $2\le A<B$, $h=\log(B/A)\ge1$, and
$H=\sum_{A<j\le B}1/j$. For $q/H\le C_0$ let $J_1,\ldots,J_q$
be iid harmonic samples on this block and
$X(t)=\sum_{v=1}^q\log|1-e(J_vt)|$.
For every circle interval $I$ of length $w>0$,
\begin{equation}\label{ra:padbound}
 \E\frac{|\{t\in I:|X(t)|>\eta h\}|}{w}
 \le\frac{C}{\eta h}
   \left\{\sqrt h+1+\log_+^2\frac1{Aw}\right\}.
\end{equation}
The estimate is uniform in the position of $I$, and therefore also
holds conditionally when $I$ is chosen by an independent low field.
\end{lemma}
\begin{proof}
For $0<\tau\le1$, all nonzero Fourier coefficients of $\phi_\tau$
are nonpositive and their absolute sum equals $A(\tau)$, by evaluation
at zero. Also $\sum_{a<j\le b}\cos(2\pi jx)/j\ge-C$:
terms up to $(6\|x\|)^{-1}$ are nonnegative and Abel summation bounds
the remaining tail. Thus, with an absolute constant uniform down
to $\tau=0$,
\[
 \frac1H\sum_{A<j\le B}\frac{\phi_\tau(jt)}j
 \le A(\tau)(1+C/H).
\]
Taking $\tau=h^{-1/2}$, using $\lambda(0)=\lambda'(0)=0$ and
$q/h\le2C_0$, gives $\E e^{\tau X(t)}\le C$.
Hence $\E X(t)_+\le C\sqrt h$.

Write $f(t)=\log|1-e(t)|$. Its mean over a period is zero and its
absolute integral is finite. Uniformly in interval position,
\[
 \frac1w\int_I f(jt)\dd t\ge-C-\log(1/(jw))\quad(jw\le1),
 \qquad
 \left|\frac1w\int_I f(jt)\dd t\right|\le\frac C{jw}
       \quad(jw\ge1).
\]
The first bound follows by integrating a logarithmic singularity;
the second uses a bounded periodic primitive of $f$.
Summing with weights $q/(Hj)$ yields
\[
 \E\frac1w\int_I X(t)\dd t
 \ge-C\{1+\log_+^2(1/(Aw))\}.
\]
Combine this with the positive-part bound to bound the expected
spatial integral of $|X|$, and apply Markov's inequality in $t$.
All logarithmic singularities are integrable.

In particular, for $h=\delta\ell$, $A\asymp e^r$, and
$w\asymp e^{-r}/r$, the right side tends to zero. A second Markov
inequality makes the exceptional fraction tend to zero in probability.
The padding count is fixed as one block; arbitrary unit padding
allocations are not conditioned on in this argument.
\end{proof}

\subsection{One weighted pair estimate}

Fix a base field $Y$ with cycle count $q_b$, independent of middle
blocks with cumulative counts $S_i$ and total $Q$. Let $W_i(t)$
be the cumulative noisy middle height, let $B_i$ be full-height
barriers, and set
\[
 \begin{aligned}
 \mathcal A_t&=\{Y(t)+W_i(t)\le B_i\ (\text{all }i),\
                         Y(t)+W_m(t)\in[y,y+1]\},\\
 Z(t)&=\E_\zeta[e^{-s\sum_i\zeta_i(t)}\ind_{\mathcal A_t}].
 \end{aligned}
\]
For a fixed base set $E$ write
\[
 w(t)=\ind_E(t)e^{sY(t)-\lambda q_b},\quad
 Z_b=\int w,\quad \mu=e^{-sy+\lambda(q_b+Q)}.
\]
For $i\ge k$, let $\widehat R_{k,i}$ be the cumulative sum of the
independent $s$-tilted reference marks in blocks $k+1,\ldots,i$,
including the independent noises in these blocks. In particular,
$\widehat R_{k,k}=0$. The reference suffix kernel is
\[
 \begin{split}
 K_k(x)=\E_{s,\zeta}\bigl[&e^{-s(x+\widehat R_{k,m}-y)}\\
  &\cdot\ind_{\{x+\widehat R_{k,i}\le B_i\ (k<i\le m),\
                   x+\widehat R_{k,m}\in[y,y+1]\}}\bigr].
 \end{split}
\]
Thus $0\le K_k\le1$. For $k=m$ the suffix is empty and
$K_m(x)=e^{-s(x-y)}\ind_{[y,y+1]}(x)$.
Put
\[
 U_k=\sup_{x\le B_k}e^{s(x-B_k)/2}K_k(x).
\]

\begin{lemma}\label{ra:pair}
Let $t,u\in E$. If $k=0$, assume $Y(t),Y(u)\le B_0$.
Suppose the one-point prefix comparisons hold at both angles with
additive error at most $\epsilon$, and
the joint suffix after $k$ has the positive comparison
\eqref{ra:positive}, or the ordinary comparison of
Lemma~\ref{rh:smooth}, with additive error $\epsilon$.
If $sB_k=a_k+\lambda(q_b+S_k)-sG_k$, then
\begin{equation}\label{ra:pairbound}
 \E[Z(t)Z(u)]\le(1+o(1))\mu^2 e^{a_k-sG_k}
       \sqrt{w(t)w(u)}\{U_k^2+C\epsilon\}.
\end{equation}
There is a second bound with the braces replaced by
\begin{equation}\label{ra:prefixbound}
 \max\{V_{k,t}(H),V_{k,u}(H)\}
       +2e^{-sH/2}+C\epsilon ,
\end{equation}
where $V_{k,t}(H)$ is the one-point tilted probability that the full
prefix constraints hold and its endpoint lies in $[B_k-H,B_k]$.
Cumulative bounded noises change the common multiplicative factor by
$1+O(L^{-9})$.
\end{lemma}
\begin{proof}
Tilt the shared prefix by $(s/2,s/2)$ and the suffix by $(s,s)$.
In the inverse tilt the prefix factor is
\[
 \exp\{(s/2)(Y(t)+W_k(t)+Y(u)+W_k(u)-2B_k)\}.
\]
The terminal inverse factors lie in $[e^{-s},1]$ on their respective
events. The unnormalized prefix half moment is bounded, by
Cauchy--Schwarz, by the geometric mean of the two one-point
$s$-moments. Its normalizer is at most
$(1+o(1))A(s)^{S_k}$.
After independent suffix comparison each of the two integrals with
its half-prefix deficit factor is at most $U_k$.
Restoring the deterministic base weights gives
\eqref{ra:pairbound}; the identity defining $B_k$ gives its
factor $e^{a_k-sG_k}$.

For \eqref{ra:prefixbound}, first split according to whether a prefix
deficit exceeds $H$. Such a deficit has exponential weight at most
$e^{-sH/2}$. Otherwise retain the constrained prefix endpoints and
apply Cauchy--Schwarz \emph{before normalizing}:
\[
 \E e^{s(P_t+P_u)/2}\ind_{A_t^p}\ind_{A_u^p}
 \le\{\E e^{sP_t}\ind_{A_t^p}
           \E e^{sP_u}\ind_{A_u^p}\}^{1/2}.
\]
The one-point comparisons bound the last geometric mean by
$\max\{V_{k,t}(H),V_{k,u}(H)\}+C\epsilon$.
Using a common upper bound here avoids taking a square root of an
additive comparison error.
The suffix can be integrated using only its normalizer. This avoids
identifying a half-tilted joint prefix marginal with an $s$-tilted law.
All discarded tests are bounded by one, so the additive comparison
errors have exactly the displayed form. The independent noises in
the two copies contribute at most $e^{Cs mL^{-10}}=1+O(L^{-9})$.
\end{proof}

For a rational strip $\{\|pt+qu\|\le a\}$ with $p,q\ne0$ and $0<a\le1/2$, each
one-coordinate section has measure $2a$. Consequently
\begin{equation}\label{ra:strip}
 \int_{\{\|pt+qu\|\le a\}}\sqrt{w(t)w(u)}\dd t\dd u
 \le2a\int w.
\end{equation}
This is Cauchy--Schwarz followed by integration of the two sections.
It holds for arbitrary measurable base sets by extending $w$ by zero.
It is the reason a small low-frequency maximum arc causes no
uncontrolled supremum of its weight.

%% file: extremes.tex
\section{The harmonic maximum}\label{re:section}

The preceding tools now apply to a model with no size constraint.
Let $B\to\infty$, $L=\log B$, $\ell=\log L$, and let $q$ be an integer
with $\kappa=q/L$ in a fixed positive compact interval.
Take $q$ iid lengths with probabilities $1/(jH_B)$, $1\le j\le B$,
and write
\[
 \begin{gathered}
 X(t)=\sum_{v=1}^q\log|1-e(J_vt)|,\qquad M_B^H=\max_tX(t),\\
 s=s_\kappa,\quad \lambda=\lambda(s),\quad
 \alpha=\alpha_\kappa,\quad m=(L+\lambda q)/s.
 \end{gathered}
\]
\begin{proposition}\label{re:harmonic}
Uniformly in these $q$,
\[
 M_B^H=m-\frac\alpha s\ell+\oo(\ell).
\]
\end{proposition}

We condition on counts in deterministic logarithmic blocks. Given these
counts, the samples in the blocks remain independent harmonic samples.
Apply the reference exponents with half the error tolerance used below,
to absorb $N\asymp L$ and the fixed terminal inverse-tilt factors.
The probability statements about reference paths used below hold on
sets of count vectors of probability tending uniformly to one, by
Lemma~\ref{rr:counts}. They are not assertions for every possible
count vector.

\subsection{The upper bound}

Fix $\eps>0$ and set $y=m-(\alpha/s-\eps)\ell$.
Choose $r=C_r\ell$. Partition $(r,L]$ into blocks of widths comparable
to $A_{\rm mid}\ell$. Within the low interval $(0,r]$, use root cutoff
$r_0=C_0\log r$ and blocks of widths comparable to
$A_{\rm low}\log r$. Chernoff bounds give, with probability tending
to one, counts between fixed positive multiples of the block widths
and upper bounds $(\kappa+\eta)H_i$ on all nonroot blocks.
The root count is $O(r_0)$, and
$q_0=\kappa r+O(\ell^{3/4})$.
Choose $\eta$ with a uniform entropy gap, as in
Lemma~\ref{rh:entropy}.

The integrated-moment argument at the smaller horizon $r$ gives,
at both tilts $\sigma=s,s/2$,
\begin{equation}\label{re:lowmoment}
 \E_{\rm low}\int_\T e^{\sigma X_0(t)}\dd t
 \le e^{\lambda(\sigma)q_0}\exp(O(\log r))
 =e^{\lambda(\sigma)q_0}L^{o(1)}.
\end{equation}
This use of Lemma~\ref{rh:integrated} needs the entropy gap, not
criticality at $\sigma$. The omitted root is bounded by
$2^{\sigma q_{\rm root}}$. In particular,
\begin{equation}\label{re:negative}
 \E_{\rm low}\int e^{sX_0-\lambda q_0}
                    \ind_{\{X_0<-D\ell\}}\dd t
 \le L^{-sD/2+o(1)},
\end{equation}
since $\lambda(s/2)\le\lambda(s)$.

Let $a_i$ be middle endpoints, $S_i$ cumulative middle counts, and
$F_i=(a_i+\lambda(q_0+S_i))/s$.
There is a fixed $K$ for which, with probability tending to one,
\begin{equation}\label{re:globalbarrier}
 X_{\le i}(t)\le F_i+K\ell\quad\text{for every }i,t.
\end{equation}
To see this, the prefix polynomial has degree at most
$CL e^{a_i}$. Bound its low field crudely by $q_0\log2$ and apply the
integrated middle moment to the remaining blocks. Bernstein's
integral inequality gives a bound $L^{C-sK}$ for each prefix
exception. There are at most $L$ prefixes, so a single large $K$
suffices. The deterministic low upper bound is covered by enlarging
$K$. This crude low estimate is used only for the generous barrier.

Choose the Fourier parameters for accuracy greater than
$\alpha_{\max}+4$ and put
\[
 \mathcal D=\{t:\|jt\|\ge e^{-r+\Delta},\
                           1\le |j|\le R\}.
\]
On $\mathcal D$, the smoothed one-point tilted law of the middle
path has its reference normalizer and an arbitrarily small polynomial
error. Enlarge the original path and endpoint inequalities by the
deterministic total noise $O(L^{-9})$; every noise realization then
contains the unnoised event, with bounded inverse-tilt endpoint factors.
Conditional on $X_0(t)=x\in[-D\ell,F_0+K\ell]$, impose
\eqref{re:globalbarrier} and terminal height at least $y-\log2$.
Under the final barrier, the endpoint range has length $O(\ell)$.
Theorem~\ref{rr:excursion} and the upper refinement implication in
Lemma~\ref{rr:counts} bound each unit terminal window by
$L^{-\alpha+\xi}$, uniformly in $x$. The middle total satisfies
$Q-\kappa(L-r)=O(\ell)$, as required for total-gap conditioning.
Inverse tilt and \eqref{re:lowmoment} show that the expected angular
measure of these constrained high points is at most
\begin{equation}\label{re:uppergood}
 C e^{-L}L^{-s\eps+\xi+o(1)}.
\end{equation}
The factor $O(\ell)$ from endpoint windows is $L^{o(1)}$.
For $x<-D\ell$, drop the path restriction and use
\eqref{re:negative}, giving
$Ce^{-L}L^{\alpha-s\eps-sD/2+o(1)}$.
Choose $D$ with $s_{\min}D/2>\alpha_{\max}+3$.

For the remaining angles, the entropy envelope must include the early
blocks. For a common $\rho<1$ it gives the prefix bound
\[
 \E e^{sX_{\le i}(t)}
 \le e^{\lambda(q_0+S_i)}e^{\rho a_i+O(r_0)}.
\]
Classify an angle by the first separated suffix, at a scale $a_i$.
Its class has measure at most $CR e^{-a_{i-1}+\Delta}$.
The suffix has its accurate normalizer. Summing the geometric series
bounds the total normalized bad-angle moment by
\[
 L^{u_2+u_3+\rho A_{\rm mid}-(1-\rho)C_r+o(1)}
\]
when $R=L^{u_2}$ and $\Delta=u_3\ell$.
Choose $C_r$ so that this is at most $L^{-H_0}$ with
$H_0>\alpha_{\max}+3$. The bad-angle high-point measure is then at most
$Ce^{-L}L^{\alpha-s\eps-H_0+o(1)}$.

Finally the harmonic degree $d=\sum J_v$ satisfies
$\E(d/B)=q/H_B\le C$. Markov's inequality therefore gives
$\Pp(d>B\ell^2)\le C\ell^{-2}\to0$.
If $M_B^H>y$, Bernstein supplies an arc of length
$c/(B\ell^2)$ on which $X\ge y-\log2$.
All prefix constraints hold on that arc on
\eqref{re:globalbarrier}. Multiply the angular measure bounds by
$CB\ell^2$ and take $\xi<s_{\min}\eps/2$.
Every exponent is negative. This proves the upper bound.

\subsection{A base set retaining the low-frequency center}

For the lower bound set $c=\alpha/s+\eps$ and $y=m-c\ell$.
All small constants in this proof are fixed after $\eps$ and the
parameter compact interval. Let $v=\kappa\lambda'(s)$.
Choose $\eta,\delta,\xi>0$ so that, uniformly,
\begin{equation}\label{re:margins}
 \begin{gathered}
 \eta<v/20,\qquad 3s\eta<1/8,\qquad
 \delta(\kappa\lambda+2s\eta)<s_{\min}\eps/8,\\
 \eta\delta<\min_\kappa(\alpha/s)/4,\qquad
 4\xi<\delta/16,\qquad 3\xi<s_{\min}\eps/16.
 \end{gathered}
\end{equation}
Put $h=\delta\ell$, $r=C_r\ell$, and $r'=r+h$.
The low region through $r$ is partitioned as in the coarse localization
proof, with its own widths of order $\log r$.
The padding $(r,r']$ is one block.
The middle range $(r',L]$ has one fixed grid of widths in $[1/2,1]$.

On a uniformly typical count event,
\[
 |q_{\rm low}-\kappa r|\le C\sqrt{r\log r},\qquad
 q_{\rm pad}=\kappa h+O(\sqrt{h\log h}),\qquad
 \max_{\rm middle}q_i\le C\ell/\log\ell.
\]
The last bound is the binomial Chernoff estimate followed by a union
over $O(L)$ blocks of bounded mean. Counts in any fixed multiple of
$\ell$ consecutive middle blocks are $O(\ell)$, also by Chernoff.
The low regularity conditions use
$\kappa_{\rm low}=q_{\rm low}/H_{\lfloor e^r\rfloor}$.

Apply Proposition~\ref{rh:localization} at the low horizon $r$.
Its optimized tilt can be replaced by $s_\kappa$ with error
$O(\log r)$: the tilt displacement is
$O(\sqrt{\log r/r})$, the derivative of
$(r+q_{\rm low}\lambda(z))/z$ vanishes at its optimum, and its
second derivative there is $O(r)$.
Thus, with probability tending to one,
\[
 M_{\rm low}=F_r+O(\log r),\qquad
 F_r=(r+\lambda q_{\rm low})/s.
\]
The low degree is at most $Cr e^r$.
Bernstein gives a measurably selected interval $I$ of length
comparable to $e^{-r}/r$, on which
$X_{\rm low}=F_r+O(\log r)$.

Take $\rho<\rho'<1$ above the uniform entropy exponent and define
\begin{equation}\label{re:strongset}
 \epsilon_0=e^{-\rho'r},\qquad
 \eta_0=e^{-r+\Delta},\qquad
 \mathcal D_*=\{t:\|jt\|\ge\epsilon_0,\
                                1\le |j|\le2R^2\}.
\end{equation}
Here $R=L^U$ includes all Fourier cutoffs and $\Delta=D_F\ell$;
their choice precedes $C_r$.
The low envelope and the measure bound for $\mathcal D_*^c$ give
\begin{equation}\label{re:lowbad}
 \E_{\rm low}\int_{\mathcal D_*^c}
             e^{sX_{\rm low}-\lambda q_{\rm low}}\dd t
 \le L^{2U-(\rho'-\rho)C_r+o(1)}.
\end{equation}
On $I$ the integrand is at least $e^r r^{-C}$.
Since $|I|\asymp e^{-r}/r$, Markov's inequality shows that
$|I\setminus\mathcal D_*|=o(|I|)$, when $C_r$ is large.

Lemma~\ref{ra:padding}, conditional on the low field and the padding
count, removes only an $o(1)$ fraction of $I$ on which
$|X_{\rm pad}|>\eta h$. Put
\[
 E=I\cap\mathcal D_*\cap\{|X_{\rm pad}|\le\eta h\},\qquad
 Y=X_{\rm low}+X_{\rm pad},\qquad q_b=q_{\rm low}+q_{\rm pad}.
\]
With probability tending to one, $|E|\ge|I|/2$ and, uniformly on $E$,
\[
 Y=F_r+[-\eta h,\eta h]+o(h),\qquad
 F_{\rm start}=(r'+\lambda q_b)/s=F_r+vh+o(h).
\]
For $w=\ind_E e^{sY-\lambda q_b}$ and $Z_b=\int w$ this implies
\begin{equation}\label{re:baseweight}
 \log Z_b\ge-(\kappa\lambda+s\eta)h+o(\ell).
\end{equation}
The cancellation $e^r|I|=r^{-1}$ is essential here. No macroscopic
measure assumption on $E$ is made.

\subsection{The path and its first moment}

Let $T=L-r'$. For middle logarithmic time $x\in[0,T]$, set
\[
 \begin{gathered}
 G_0=(v-2\eta)h,\qquad G_T=c\ell-\eta h,\\
 G(x)=G_0+(G_T-G_0)x/T
       +a_*\{(\min(x,T-x)+\tau)^{1/3}-\tau^{1/3}\}.
 \end{gathered}
\]
Here $a_*>0$ is fixed and $\tau$ is large enough that
$|G'|<1/(2s)$ uniformly, away from its harmless midpoint corner.
The full barrier is
$B_i=(a_i+\lambda(q_b+S_i))/s-G(a_i-r')$.
Its initial gap above $Y$ lies between $\eta h+o(h)$ and
$3\eta h+o(h)$, and its final gap above $y$ is $\eta h$.

Use the fixed noisy functional of Lemma~\ref{ra:pair} on $E$ and
write $Z_E=\int_E Z(t)\dd t$.
For an observation angle $t$, let $\widehat W_i$ be the sum of the
independent $s$-tilted reference marks in the first $i$ middle blocks,
plus their cumulative noises. The reference marks have the same fixed
block counts as the actual middle field. Define
\begin{equation}\label{re:reference-functional}
 \begin{split}
 p_t=\E_{s,\zeta}\bigl[&e^{-s(Y(t)+\widehat W_m-y)}
       \ind_{\{Y(t)+\widehat W_i\le B_i\ (1\le i\le m),\
                    Y(t)+\widehat W_m\in[y,y+1]\}}\bigr].
 \end{split}
\end{equation}
The integrand lies in $[0,1]$. Its exponential factor lies in
$[e^{-s},1]$ on the path event, so the reference excursion exponent
also bounds $p_t$. We retain this weighted expectation, rather than
replacing it by the unweighted path probability, in both moments.
The curved lower excursion estimate and its refinement implication
give, on typical middle count vectors, simultaneously for the
permitted base heights,
\[
 p_t\ge L^{-\alpha-\xi}.
\]
A fixed inner terminal interval and a slightly stronger barrier handle
every bounded noise realization. The unused last spacing is required
to be smaller than a fixed fraction of the positive terminal buffer;
its exceptional probability tends to zero.
For the subsequent uniform upper comparisons a common maximum-gap
event of size $O(\ell)$ suffices, including empty unit blocks.
These count-typicality statements are obtained before the counts are
fixed in the moment calculation.

To verify the first-moment normalization, put
$N_t=\prod_i B_i^{\rm mom}(s;t)^{q_i}$, where
$B_i^{\rm mom}(s;t)$ is the one-sample harmonic moment in middle
block $i$; the superscript distinguishes it from the height barrier.
Exponential change of measure gives the exact identity
\[
 \E_{\rm mid} Z(t)
 =e^{-sy+sY(t)}N_t\,
   \E_{{\rm tilted},t,\zeta}
      [e^{-s(Y(t)+W_m(t)-y)}\ind_{\mathcal A_t}].
\]
Indeed the factor $e^{-s\sum_i\zeta_i}$ in $Z(t)$ combines with
the inverse mark tilt to give this terminal weight. The one-point
comparison on $\mathcal D_*$ gives $N_t/A(s)^Q=1+o(1)$ and
compares the displayed bounded test with \eqref{re:reference-functional}.
Consequently
\begin{equation}\label{re:first}
 \E_{\rm mid}Z_E=(1+o(1))\mu A_E,\qquad
 \mu=e^{-sy+\lambda q},\quad
 A_E=\int w(t)p_t\dd t\ge Z_bL^{-\alpha-\xi}.
\end{equation}
The comparison error is chosen smaller than the last lower bound
after division by $Z_b$.
If $Z_E>0$, the actual harmonic maximum is at least $y-O(L^{-9})$.

\subsection{Pairs with independent suffixes}

Call a pair ordinary far if every nonzero integer vector of norm at
most $R$ has phase at least $\eta_0$.
All its middle blocks satisfy the ordinary product comparison.
A pair near either true diagonal, $u=t+d$ or $u=-t+d$,
also has whole-middle product comparison when
\[
 e^{-r'}\ell^4\le |d|\le\eta_0.
\]
Indeed it lies in $\mathcal D_*^2$, and the last choice of $C_r$
makes $|d|/\epsilon_0$ satisfy Lemma~\ref{ra:phase}.
On these two types of pairs, reference paths are independent given
the base and counts. Their integral contributes at most
\[
 (1+o(1))\mu^2 A_E^2+O(\mu^2Z_b^2L^{-J}),
\]
which is $(1+o(1))\mu^2 A_E^2$ for sufficiently large $J$.

For the remaining true-diagonal pairs, $0<|d|<e^{-r'}\ell^4$,
let $k$ be the first middle endpoint with $e^{a_k}|d|\ge\ell^4$.
If none exists use the final block. The suffix after $k$ has the
positive comparison, or is empty. Each class has one-coordinate
sections of measure at most
\begin{equation}\label{re:diagonalwidth}
 C\ell^4e^{-a_{k-1}}.
\end{equation}
At the last class this means all separations below the previous
threshold. Exact diagonals have zero two-dimensional measure.
Because the grid widths are bounded, multiplication by $e^{a_k}$
costs only $O(\ell^4)=L^{o(1)}$.

Put $b_\ell=\ell^6$ and take a deficit cutoff $H=\ell^P$ with fixed
large $P$. If $a_k-r'\le b_\ell$, the uniform shifted suffix bound
of Lemma~\ref{rr:shifts} gives $U_k\le L^{-\alpha+\xi}$:
larger deficits are suppressed by $e^{-sH/2}$, and smaller ones
are within the polynomial-logarithmic shift range. Apply
\eqref{ra:pairbound} and \eqref{ra:strip}, sum the polynomially many
early classes, and divide by \eqref{re:first} squared.
The result is at most
\begin{equation}\label{re:early}
 L^{o(1)}e^{-sG_0}Z_b^{-1}L^{4\xi}
 \le L^{-\delta+3s\eta\delta+4\xi+o(1)}=o(1).
\end{equation}

If $L-a_k\le b_\ell$, retain the constrained prefix in
\eqref{ra:prefixbound}. Its initial gap and terminal deficit are in
the shifted excursion range, so both one-point prefix probabilities
are at most $L^{-\alpha+\xi}$. The same integration gives
\begin{equation}\label{re:late}
 L^{o(1)}e^{-sG_T}Z_b^{-1}L^{\alpha+3\xi}
 \le L^{-s\eps+\delta\kappa\lambda
                   +2s\eta\delta+3\xi+o(1)}=o(1).
\end{equation}
For the other splits the curved part of $G$ is at least
$c_0\ell^2$. Dropping all path costs in the pair bound gives an
exponential factor $e^{-c_1\ell^2}$, which absorbs the $O(L)$
classes and every polynomial normalization loss.
The small errors in all pair comparisons are negligible by increasing
$J$ at the outset. Initial near pairs are included in
\eqref{re:diagonalwidth}; there is no missing macroscopic first class.

\subsection{Other rational pairs}

It remains to control pairs outside the two $\eta_0$ diagonal
neighborhoods with a relation $\|pt+qu\|\le\eta_0$.
Both coordinates of this relation are nonzero on $\mathcal D_*^2$.
Let
\[
 a_0\ge\rho'r+\log R+J_0\ell
\]
be a low coarse-block boundary. The determinant geometry and
Lemma~\ref{ra:resonance} give a fixed contraction $1-d_*$ for each
half-tilted sample in the low blocks between $a_0$ and $r$.
Their typical counts give the product contraction
\begin{equation}\label{re:contraction}
 \exp[-c_{\rm gap}(r-a_0)],\qquad c_{\rm gap}>0.
\end{equation}

Here the low and padding samples must be averaged before the final
base is fixed. For the earlier low prefix let
$K_{\rm early}(t)=\E e^{sX_{\rm early}(t)}$ under its retained coarse
count law. Cauchy--Schwarz bounds its joint half moment by
$\sqrt{K_{\rm early}(t)K_{\rm early}(u)}$.
For a rational strip, \eqref{ra:strip} reduces the integral to its
width times $\int K_{\rm early}$. The nested low integrated estimate
gives
\[
 \int K_{\rm early}\le A(s)^{q_{\rm early}}L^{o(1)}.
\]
Later low blocks supply \eqref{re:contraction}.
The padding and middle one-point normalizers are accurate on
$\mathcal D_*$; Cauchy--Schwarz controls their half moments.

For completeness, classify the middle suffix by the first index $k$
at which the minimum rational phase is at least $e^{-a_k+\Delta}$.
For $k\ge1$ its area and each section are bounded by
$CR^2e^{-a_{k-1}+\Delta}$. The first class $k=0$ has width at most
$CR^2\eta_0$; multiplying by $e^{r'}$ costs
$CR^2e^{\Delta+h}$. This padding factor is allowed here.
The final class has empty suffix. Here is the normalization and the
order of averaging in detail. Let $X_{\le a_k}$ denote the full
prefix, including the low and padding samples, and set
\[
 H_k(t,u)=\E_{\le a_k}
       e^{s\{X_{\le a_k}(t)+X_{\le a_k}(u)\}/2},
 \qquad q_k=q_b+S_k.
\]
The expectation here is over the prefix samples with their retained
block counts fixed, before the low field or its selected arc is fixed.
On both path events each full prefix is at most $B_k$, up to the
bounded noises. The two terminal inequalities, followed by one use
of this prefix bound, leave a half tilt on the full prefix and a full
tilt on each suffix. The suffix normalizer is
$(1+o(1))A(s)^{2(q-q_k)}$. Consequently the bound for the fully
averaged contribution at $(t,u)$ is
\[
 (1+o(1))e^{-2sy+sB_k}A(s)^{2(q-q_k)}H_k(t,u)
 =(1+o(1))\mu^2e^{a_k-sG_k}\frac{H_k(t,u)}{A(s)^{q_k}}.
\]
The noise factors are $1+O(L^{-9})$. Dropping the selected-base
restrictions only increases this nonnegative expectation.

Write $q_{\rm early}$ for the count up to $a_0$. The earlier
Cauchy--Schwarz bound, the contraction in \eqref{re:contraction},
and the accurate later normalizers imply, on each such class,
\[
 \frac{H_k(t,u)}{A(s)^{q_k}}
 \le C e^{-c_{\rm gap}(r-a_0)}
   \frac{\sqrt{K_{\rm early}(t)K_{\rm early}(u)}}
        {A(s)^{q_{\rm early}}}.
\]
This inequality is used only where both angles are in $\mathcal D_*$.
To bound its integral we then enlarge the domain to the union of
rational strips defining the class. If its section-width bound is
$w_k$, equation~\eqref{ra:strip} and the integrated early moment give
an integral at most $Cw_k L^{o(1)}e^{-c_{\rm gap}(r-a_0)}$.
Since $G_k\ge0$, summing the $O(L)$ classes with the above bounds
on $e^{a_k}w_k$ proves that their expected conditional contribution
is bounded by
\begin{equation}\label{re:nontrue}
 \mu^2 L^{C_{\rm arith}+o(1)}
       e^{-c_{\rm gap}(r-a_0)},\qquad
 C_{\rm arith}=2U+D_F+\delta+3.
\end{equation}
A larger fixed constant would give the same conclusion.
No pointwise envelope is asserted after conditioning arbitrary
low unit allocations.

Choose $C_r$ to make \eqref{re:nontrue} smaller than
$\mu^2$ times the deterministic lower bound for $A_E^2$ in
\eqref{re:first}, by a fixed negative power of $L$.
Markov's inequality over the low/padding configuration then makes
this contribution conditionally negligible on a set of probability
tending to one. This averaging is legitimate because that denominator
has a deterministic lower bound on the declared good base event.
Explicitly, \eqref{re:baseweight} and \eqref{re:first} give
\[
 A_E\ge L^{-d_\kappa+o(1)},\qquad
 d_\kappa=\alpha+\xi+\delta(\kappa\lambda+s\eta).
\]
The $o(1)$ can be bounded deterministically and uniformly after the
good base events are chosen. The budget below makes the expectation
of the nontrue-pair contribution, divided by this deterministic lower
bound for $\mu^2 A_E^2$, at most $L^{-\chi}$ for a fixed $\chi>0$.
Markov's inequality at threshold $L^{-\chi/2}$ yields the asserted
conditional negligibility. No average of a random reciprocal $A_E^{-2}$
is needed.

\subsection{Constant order and completion}

Here is one common choice for a compact parameter interval.
First choose $\eta,\delta,\xi$ as in \eqref{re:margins}.
Take
\[
 J>2\alpha_{\max}
       +2\delta\max_\kappa(\kappa\lambda+s\eta)+30.
\]
Choose the phase-cell exponents and the ordinary Fourier comparison
at higher accuracy if needed; let $U$ cover both frequency cutoffs.
Choose $\Delta=D_F\ell$ and $J_0$ next. These choices do not depend on
$C_r$. The entropy gap gives fixed $\rho<\rho'<1$.
The continuous nonzero coefficient in Lemma~\ref{ra:resonance}
gives a common $c_{\rm gap}>0$.
Finally choose $C_r$ so large that
\begin{align}\label{re:budget}
 (\rho'-\rho)C_r&>2U+20,\\
 (1-\rho')C_r&>D_F+\max(K_{\rm phase},U)+10,\nonumber\\
 c_{\rm gap}\{(1-\rho')C_r-U-J_0\}
 &>C_{\rm arith}+2\alpha_{\max}+2\xi
       +2\delta\max_\kappa(\kappa\lambda+s\eta)+30.\nonumber
\end{align}
Here $K_{\rm phase}$ is the exponent in Lemma~\ref{ra:phase}.
The second inequality ensures both its discrepancy condition and
$R\eta_0\ll\epsilon_0$. Rounding $a_0$ to a low endpoint costs
$O(\log r)=o(\ell)$. All coefficients of $C_r$ are positive, so
the choices are compatible. The curve offset and reference buffer
constants may be fixed afterwards.

On the resulting good base/count events, the far-pair integral is
at most $(1+o(1))\mu^2A_E^2$. Equations
\eqref{re:early}--\eqref{re:nontrue} make all other contributions
$o(\mu^2A_E^2)$. Thus
\[
 \frac{\E_{\rm mid}Z_E^2}{(\E_{\rm mid}Z_E)^2}\longrightarrow1.
\]
Chebyshev's inequality gives $\Pp_{\rm mid}(Z_E=0)\to0$.
Average the exceptional count and base events, whose probabilities
tend uniformly to zero. This proves the lower bound and
Proposition~\ref{re:harmonic}.
The conclusion uses no independence of fields on disjoint arcs.

\begin{proof}[Proof of Theorem~\ref{rn:main}]
Write $L=\log n$, $\ell=\log L$, $a=\log b$ and
$\kappa'=q/a$. Under the binomial/harmonic law of
Theorem~\ref{rc:transfer}, the long count is $O_{\Pp}(T)$, where
$T\sim2\log\ell$. Thus $\kappa'$ lies in a fixed enlarged positive
compact interval with probability tending uniformly to one.
Conditional on $q$, the short configuration is exactly the harmonic
model of Proposition~\ref{re:harmonic}. Its uniform conclusion,
averaged over $q$ on this compact interval, gives
\[
 M_b=m^H_{b,q}-c(\kappa')\log a+o_{\Pp}(\ell).
\]
Here $\log a\sim\ell$, and the exceptional count event has
probability $o(1)$. The quantities in this display are all functions
of the short configuration and its count. The total variation bound
\eqref{rc:joint-tv} therefore transfers the same statement to
$\Pp_{n,k}$.

Proposition~\ref{rc:restoration} gives
$M_n-M_b=o_{\Pp}(\ell)$ and
$m_{n,k}-m^H_{b,q}=o_{\Pp}(\ell)$. It also gives
$c(\kappa')\log a-c(\kappa)\ell=o_{\Pp}(\ell)$, since $c$ is
continuous. Consequently
\begin{align*}
 M_n-m_{n,k}+c(\kappa)\ell
 ={}&(M_n-M_b)
   +\{M_b-m^H_{b,q}+c(\kappa')\log a\}\\
   &+(m^H_{b,q}-m_{n,k})
   +\{c(\kappa)\ell-c(\kappa')\log a\}
   =o_{\Pp}(\ell).
\end{align*}
Each of these four errors is uniform for $k/L\in K$.
The union bound with tolerance $\eps\ell/4$ proves the supremum
statement in the theorem.
\end{proof}

%% file: consequences.tex
\section{Ewens random centering and two matchings}\label{rn:consequences}

\begin{corollary}\label{rn:ewens}
Under the Ewens law with any fixed $\theta>0$, let
$a_\theta=\lambda(s_\theta)/s_\theta$.
Then
\begin{align}\label{rn:ewens-center}
 M_n&=m_{n,K_n}-c(\theta)\log\log n
                         +\oo(\log\log n),\\
 &=v(\theta)\log n+a_\theta(K_n-\theta\log n)
      -c(\theta)\log\log n+\oo(\log\log n).\nonumber
\end{align}
Moreover,
\[
 \left(\frac{K_n-\theta\log n}{\sqrt{\theta\log n}},
 \frac{M_n-v(\theta)\log n}{a_\theta\sqrt{\theta\log n}}\right)
 \ \Longrightarrow\ (G,G),
\]
where $G$ is standard normal.
\end{corollary}
\begin{proof}
The cycle-count generating function is
\[
 \E_\theta u^{K_n}
 =\frac{\theta u(\theta u+1)\cdots(\theta u+n-1)}
        {\theta(\theta+1)\cdots(\theta+n-1)}
 =\prod_{j=1}^n\left(1+\frac{\theta}{\theta+j-1}(u-1)\right).
\]
Thus $K_n$ has the law of a sum of independent Bernoulli variables
with probabilities $\theta/(\theta+j-1)$.
Their mean and variance are $\theta\log n+O(1)$.
Expanding the characteristic function of the centered sum at
$t/\sqrt{\theta\log n}$ gives the normal limit: the sum of the
third-order remainders is $O((\log n)^{-1/2})$ since each centered
third absolute moment is bounded by a constant times its Bernoulli
probability. In particular $K_n/\log n\to\theta$ in probability.

Apply Theorem~\ref{rn:main} on a compact neighborhood of $\theta$
and average its uniform conditional failure bound over $K_n$.
Continuity of $c$ gives the first line of \eqref{rn:ewens-center}.
Differentiating the critical equation gives
\[
 s_\kappa'=-\frac1{\kappa^2s_\kappa\lambda''(s_\kappa)},\qquad
 v'(\kappa)=\frac{\lambda(s_\kappa)}{s_\kappa},\qquad
 v''(\kappa)=-\frac1{\kappa^3s_\kappa^3\lambda''(s_\kappa)}.
\]
The last derivative is bounded on that neighborhood. On the event $K_n/L$ lies in that neighborhood, whose complement
has probability tending to zero,
\[
 \left|Lv(K_n/L)-Lv(\theta)
       -a_\theta(K_n-\theta L)\right|
 \le C(K_n-\theta L)^2/L=O_{\Pp}(1).
\]
Since $\ell\to\infty$, this proves the second line.
Also $\lambda(s)>0$ for $s>0$, because
$\lambda(0)=\lambda'(0)=0$ and $\lambda''>0$, so $a_\theta>0$.
The difference between the two normalized coordinates in the
corollary tends to zero in probability, since
$\ell/\sqrt L\to0$. The joint limit follows from the count CLT.
This asserts joint weak convergence, not convergence of second moments.
At order-one accuracy the quadratic Taylor term could no longer
be discarded.
\end{proof}

\begin{remark}\label{rn:cg-comparison}
For a uniform permutation, put $s=s_1$, $a=\lambda(s)/s$ and $b=1/s$.
The critical equation gives $v(1)=a+b$. Corollary~\ref{rn:ewens}
therefore reads
\[
 M_n=aK_n+b\log n-c(1)\log\log n+\oo(\log\log n).
\]
This is Conjecture~1.6 of Cook and Gu \cite{CG}, with
$\theta_*=s$, $c_1=a$, $c_2=b$ and $c_3=c(1)$.
Their cumulant agrees with \eqref{rn:cumulant} by the gamma
duplication formula, and both definitions of $\gamma$ use independent
standard Brownian motions. Their characteristic polynomial
$Q_n(t)=\det(e(t)I-P_n)$ satisfies
$Q_n(t)=e(nt)F_n(e(-t))$, so the maxima being compared are identical.
The Taylor error used above is $O_{\Pp}(1)$; the remainder in the
maximum theorem itself is only $o_{\Pp}(\log\log n)$.
\end{remark}

Fix a perfect matching $a$ on $2n$ labels and choose another perfect
matching $b$ uniformly. The union of their edges has alternating
components; a common edge is a component on two vertices.
Let $C_j$ count components on $2j$ vertices and let $K=\sum_jC_j$.

\begin{corollary}\label{rn:matchings}
Conditional on $K=k$, the logarithmic maximum of the characteristic
polynomial of the permutation $ab$ satisfies
\[
 \log\|\det(I-zP_{ab})\|_\infty
 =2m_{n,k}-2c(k/\log n)\log\log n
                     +o_{\Pp_{n,k}}(\log\log n),
\]
uniformly for $k/\log n$ in a fixed positive compact interval.
Unconditionally its expansion is twice \eqref{rn:ewens-center}
with $\theta=1/2$ and $K_n$ replaced by $K$.
\end{corollary}
\begin{proof}
We recall the exact encoding underlying \cite{Lu}, including its
normalization. Label the fixed matching edges by $1,\ldots,n$.
A binary coloring in which both matchings join opposite colors has
two choices on each alternating component, hence $2^k$ choices
for a profile with $k$ components.
For a colored matching, orient each fixed edge according to its colors,
naming its endpoints
$\operatorname{out}(i),\operatorname{in}(i)$.
There is a unique permutation $\pi$ such that $b$ matches
$\operatorname{out}(i)$ with $\operatorname{in}(\pi(i))$.
Conversely, an arbitrary choice of the $n$ orientations together with
a permutation $\pi$ determines a colored matching.
The cycles of $\pi$ are exactly the alternating components.
Therefore
\[
 \#\{b:\text{profile }c\}
 =2^{n-k}n!\prod_j\frac{j^{-c_j}}{c_j!},
 \qquad \sum_jjc_j=n,\quad \sum_jc_j=k.
\]
The normalization is
$(2n-1)!!=2^n n!h_n(1/2)$.
Thus the component profile has the Ewens$(1/2)$ law, and its
conditional law given $K=k$ is \eqref{rc:exact-law}.
Uniform conjugation of a fixed matching also gives this uniform
matching law, because all fibers are cosets of its centralizer.

On the two color classes the product $ab$ acts as $\pi$ and
$\pi^{-1}$, respectively. Hence a $2j$-vertex component contributes
two $j$-cycles, including two fixed points when $j=1$, and
\[
 \det(I-zP_{ab})=\prod_j(1-z^j)^{2C_j}.
\]
The logarithmic maximum is exactly twice that of the corresponding
size-$n$ cycle polynomial with profile $(C_j)$. Apply Theorem~\ref{rn:main}
and Corollary~\ref{rn:ewens}. The conditioning variable counts
graph components; the permutation $ab$ itself has $2k$ cycles.
\end{proof}